\documentclass[reqno]{amsart}
\usepackage{microtype}

\usepackage[shortlabels]{enumitem}
\setlist[enumerate]{label={\arabic*.}}
\setlist[description]{font=\normalfont\slshape}

\usepackage{amssymb}
\usepackage{bm}  
\usepackage[dvipsnames]{xcolor}
\usepackage
[colorlinks=true,linkcolor=Maroon,citecolor=OliveGreen,backref]
{hyperref}
\usepackage{url}
\usepackage[abbrev,shortalphabetic]{amsrefs}  

\usepackage{mathtools}
\mathtoolsset{showonlyrefs}

\newtheorem{theorem}{Theorem}[section]
\newtheorem*{theorem*}{Theorem}
\newtheorem{lemma}[theorem]{Lemma}
\newtheorem{proposition}[theorem]{Proposition}

\theoremstyle{definition}
\newtheorem{remark}[theorem]{Remark}
\newtheorem*{remark*}{Remark}
\newtheorem{conjecture}[theorem]{Conjecture}
\newtheorem{fact}[theorem]{Fact}

\newcommand{\eps}{\varepsilon}
\renewcommand{\phi}{\varphi}
\renewcommand{\bar}{\overline}
\renewcommand{\tilde}{\widetilde}
\renewcommand{\hat}{\widehat}

\newcommand\opr[1]{\operatorname{#1}}

\def\Q{\mathbf{Q}}
\def\C{\mathbf{C}}
\def\Z{\mathbf{Z}}
\def\F{\mathbf{F}}

\def\VV{\mathcal{V}}  
\def\UU{\mathcal{U}}  

\def\M{\opr{M}}  
\def\GL{\opr{GL}}
\def\Spec{\opr{Spec}}
\def\Gal{\opr{Gal}}

\def\P{\mathbf{P}}
\def\E{\mathbf{E}}

\def\supp{\opr{supp}}
\def\codim{\opr{codim}}
\def\row{\opr{row}}

\def\TV{d_{\mathrm{TV}}}

\def\muhat{\hat{\mu}}
\def\nuhat{\hat{\nu}}
\def\minuso{\setminus\{0\}}
\def\bmu{\bm{\mu}}

\def\OO{\mathcal{O}}
\def\pp{\mathfrak{p}}

\newcommand\br[1]{{\left(#1\right)}}
\newcommand\floor[1]{\left\lfloor{#1}\right\rfloor}

\newcommand\pfrac[2]{\br{\frac{#1}{#2}}}

\newcommand\indic[1]{1_{#1}}  

\usepackage{todonotes}

\begin{document}

\title[The characteristic polynomial of a random matrix]{The characteristic polynomial of a\\random matrix}

\author{Sean Eberhard}
\address{Sean Eberhard, Centre for Mathematical Sciences, Wilberforce Road, Cambridge CB3~0WB, UK}
\email{eberhard@maths.cam.ac.uk}

\thanks{SE has received funding from the European Research Council (ERC) under the European Union’s Horizon 2020 research and innovation programme (grant agreement No. 803711).}

\begin{abstract}
  Form an $n \times n$ matrix by drawing entries independently from $\{\pm1\}$ (or another fixed nontrivial finitely supported distribution in $\Z$)
  and let $\phi$ be the characteristic polynomial.
  We show, conditionally on the extended Riemann hypothesis,
  that
  with high probability $\phi$ is irreducible and $\Gal(\phi) \geq A_n$.
\end{abstract}

\maketitle


\section{Introduction}

Let $\mu$ be a fixed nontrivial finitely supported measure on $\Z$, say $\mu = (\delta_{-1} + \delta_1) /2$.
Let $\M_n(\mu)$ denote the distribution of $n \times n$ matrices whose entries are independent with distribution $\mu$.
Discrete random matrix theory studies the statistical algebraic properties of a random matrix $M \sim \M_n(\mu)$ for $\mu$ fixed and $n$ large.
A foundational result of Koml\'os~\cite{komlos} asserts that $M$ is nonsingular with probability $1 - o(1)$
(and after a long series of improvements it was finally established by Tikhomirov~\cite{tikhomirov} (in the $\mu = (\delta_{-1} + \delta_{+1})/2$ case)
that $M$ is singular with the expected probability $2^{-n + o(n)}$.
In another direction, Tao and Vu~\cite{tao-vu-simple-spec} proved that random symmetric matrices have no repeated eigenvalues,
a property which is of interest in the graph isomorphism problem.
A broad generalization of these algebraic properties of random matrices was conjectured by Babai (already in the 1970s) and Vu--Wood (2009) (see \cite{vu-survey}*{Conjecture~11.1}).

\begin{conjecture}[Babai, Vu--Wood]
  \label{main-conjecture}
  Let $M \sim \M_n(\mu)$ for fixed $\mu$ and $n \to \infty$
  and let $\phi(t) = \det(t - M)$ be the characteristic polynomial.
  Then $\phi$ is irreducible with high probability.\footnote{Throughout the paper we use the phrase ``with high probability'' to mean with probability $1 - o(1)$ as $n \to \infty$.
    Any stronger bounds are made explicit.}
\end{conjecture}

In this paper we will prove the above conjecture under certain conditions,
specified in the following two theorems.

\begin{theorem}
  \label{main-four-prime-case}
  Assume there is a product of four distinct primes $N = p_1 p_2 p_3 p_4$ such that $\mu \bmod N$ is uniform.
  Then with high probability $\phi$ is irreducible and $\Gal(\phi) \geq A_n$.
\end{theorem}

\begin{theorem}
  \label{main}
  Assume the extended Riemann hypothesis \textup{(ERH)}\footnote{The Riemann hypothesis for a number field $K$ asserts that the Dedekind zeta function $\zeta_K(z)$ has no zeros with $\Re z > 1/2$. The extended Riemann hypothesis asserts that the Riemann hypothesis holds for all number fields $K$.}.
  Then with high probability $\phi$ is irreducible and $\Gal(\phi) \geq A_n$.
  Quantitatively,
  \begin{align*}
    \P(\phi~\textup{irreducible})
     & = 1 - O(e^{-cn}),
    \\
    \P(\Gal(\phi) \geq A_n)
     & = 1 - O(e^{-cn^{1/2}}),
  \end{align*}
  where $c$ and the implicit constants depend on $\mu$.%
  \footnote{
  The dependence on $\mu$ is a little complicated.
  Assume $\mu$ is supported on $[-H, H]$ and $\mu(x) \leq 1 - \alpha < 1$ for all $x \in \Z$.
  Let $\bar\mu$ be the mean of $\mu$ and assume \emph{either} $\bar\mu=0$ or $|\bar\mu| \geq \delta > 0$.
  Then the constants depend only on $H$, $\alpha$, and $\delta$.
  The dependence on $\delta$ is particularly silly and can almost certainly be eliminated with more work.}
\end{theorem}

Distinguishing between $A_n$ and $S_n$ remains open,
but presumably $\Gal(\phi)=S_n$ with high probability.

\subsection{Method}

Conjecture~\ref{main-conjecture} parallels a conjecture of Odlyzko and Poonen~\cite{OP} about random polynomials.
The Odlyzko--Poonen conjecture predicts that a random polynomial of degree $n$ with coefficients drawn independently at random from a nontrivial finitely supported measure $\mu$ is irreducible with probability $1-o(1)$.
This conjecture has recently been settled under conditions analogous to those above
by Bary-Soroker and Kozma~\cite{BSK}
and Breuillard and Varj\'u~\cite{BV},
respectively.
Both papers additionally prove that the Galois group is at least $A_n$, as above.

Both \cite{BSK} and \cite{BV} use a local-to-global principle to reduce the problem to a question about random polynomials mod $p$.
The difference between the methods is mainly the size of $p$.
The method of \cite{BSK} uses bounded primes, specifically the four primes specified in the hypothesis.
The reduction of the polynomial modulo any of these primes is uniform and hence its factorization is well-understood.
Moreover, the factorizations modulo different primes are independent.
One then proves that with high probability there can be no factorization upstairs that is compatible with the four factorizations downstairs.

This method depends crucially on the following result of Pemantle, Peres, and Rivin~\cite{pemantle-peres-rivin}:
if $\pi_1, \pi_2, \pi_3, \pi_4 \in S_n$ are chosen uniformly at random
then with probability bounded away from zero there do not exist $\sigma_1, \sigma_2, \sigma_3, \sigma_4 \in S_n$
such that $\langle \pi_1^{\sigma_1}, \pi_2^{\sigma_2}, \pi_3^{\sigma_3}, \pi_4^{\sigma_4} \rangle < S_n$.
It is known that 4 is the smallest number acceptable in such a statement (see \cite{EFG-invariable}).
This explains why we require four primes in Theorem~\ref{main-four-prime-case}.
On the other hand, in a recent tour-de-force using many additional tools from analytic number theory and sieve theory,
Bary-Soroker, Koukoulopoulos, and Kozma~\cite{BSKK}
substantially expanded the method of \cite{BSK} to include a much wider class of coefficient distributions.
Using their method, it may be possible to relax the hypotheses of Theorem~\ref{main-four-prime-case} considerably;
for example, it may be possible to allow any uniform distribution on an interval of length at least 35 as in \cite{BSKK}.

On the other hand the method of \cite{BV} uses large primes.
It is a consequence of the prime ideal theorem that
for any fixed polynomial $\phi$
the number of roots of $\phi \bmod p$ is, on average over $p$,
equal to the number of distinct irreducible factors of $\phi$.
Hence if one can show that $\phi$ has only one root mod $p$ on average over large $p$ then it follows that $\phi$ is irreducible
or, possibly, a proper power,
an intuitively remote possibility that must be ruled out specially.
The size of $p$ is determined by the error term in the prime ideal theorem, which is how ERH comes into play.

In this paper we adapt both methods to the case of a characteristic polynomial $\phi(t) = \det(t - M)$.
At the level of detail just explained, nothing changes.
However, the local problem is completely different.
Whereas for random polynomials one needs to understand the roots of random polynomials mod $p$,
for the characteristic polynomial of a random matrix one needs to understand the eigenvalues of a random matrix mod $p$.

For the proof of Theorem~\ref{main-four-prime-case} one may use the fact that the characteristic polynomial of a uniformly random matrix mod $p$ is little different from a uniformly random polynomial,
apart from the smallest irreducible factors.

For the proof of Theorem~\ref{main}, one needs to understand the number of eigenvalues of $M$ mod $p$ for large $p$, say $p \approx e^{cn}$.
This problem has been studied in random matrix theory,
but we need a more general version of the closest result in the literature.

\subsection{The local problem}

Since $\lambda$ is an eigenvalue of $M$ if and only if $M - \lambda$ is singular, as far as the first moment goes it suffices to understand the singularity probability of a random matrix with independent (but not identically distributed) entries.

We will prove the following general result, of independent interest.
Let $q$ be an arbitrary prime power.
A measure $\mu$ on $\F_q$ is called \emph{$\alpha$-balanced} if
$\mu(x+H) \leq 1 - \alpha$ for every $x \in \F_q$ and every proper subgroup $H < \F_q$.
Let $\bmu = (\mu_{ij})_{1 \leq i,j\leq n}$ be a matrix of $\alpha$-balanced measures on $\F_q$, and let $\M(\bmu)$ denote the distribution of matrices $M \in \M_n(q)$ with independent entries and $M_{ij} \sim \mu_{ij}$.

\begin{theorem}
  \label{thm:prob-M-nonsingular}
  For $M \sim \M(\bmu)$,
  \[
    \P(M~\textup{nonsingular})
    = \prod_{i=1}^\infty (1-1/q^i) + O_\alpha(e^{-c\alpha n}),
  \]
  where $c>0$ is an absolute constant.
\end{theorem}

For the purpose of proving Theorem~\ref{main}, we do not need the full strength here.
It would suffice to assume that $q$ is prime
and that each $\mu_{ij}$ is a translate of a common measure $\mu$,
and to establish a main term of the form $1-1/q+O(1/q^2)$,
but the full generality is interesting and not much more difficult.

Theorem~\ref{thm:prob-M-nonsingular}
is close to a result of Maples~\cite{maples}
(which is based on a method of Tau and Vu~\cite{tao-vu-singularity})
and Nguyen--Paquette~\cite{nguyen-paquette}*{Appendix~A}.
Our statement is more special in that we consider only constant $\alpha$,
but more general in that we do not assume the entries are identically distributed.
To prove this we follow \cites{maples, nguyen-paquette} closely,
making changes where necessary
and simplifications where possible.
(The recent paper \cite{LMN} is also relevant, but does not provide an exponential error bound, which is crucial for us.)

Theorem~\ref{thm:prob-M-nonsingular} is sufficient for proving the irreducibility part of Theorem~\ref{main},
but to prove that $\Gal(\phi) \geq A_n$ we must also bound the correlations of eigenvalue events.
For $\lambda \in \F_q$ let $E_\lambda$ be the event that $\lambda$ is an eigenvalue of $M$.
The following result is the most novel and most technical contribution of the paper.

\begin{theorem}
  \label{thm:correlation-bound}
  Assume $m < c \log n$ and $C \leq \log q < c (\alpha n)^{1/2} / m^{3/2}$.
  For $M \sim \M(\bmu)$ and distinct $\lambda_1, \dots, \lambda_m \in \F_q$,
  \[
    \P(E_{\lambda_1} \cap \cdots \cap E_{\lambda_m})
    \leq (q-1)^{-m} + O_\alpha \br{\exp(-c (\alpha n)^{1/2} / m^{3/2})}.
  \]
  Here $c>0$ and $C$ are absolute constants.
\end{theorem}

\subsection{Dependence on CFSG}

The $\Gal(\phi) \geq A_n$ part of Theorem~\ref{main} is proved
by showing that $\Gal(\phi)$ is $m$-transitive with high probability for any constant $m$, and
then appealing to the following fact.

\begin{fact}
  \label{6trans}
  There is a constant $m$
  such that if $G \leq S_n$ and $G \ngeq A_n$
  then $G$ is at most $m$-transitive.
\end{fact}

This follows from the classification of finite 2-transitive groups,
which is a deep consequence of the classification of finite simple groups (CFSG),
and depends on detailed knowledge of the finite simple groups.
One may take $m = 5$, or even $m = 3$ for $n > 24$.
Alternatively, it follows from the Schreier conjecture, another hopelessly deep consequence of CFSG.
See \cite{dixon--mortimer}*{Section~7.3}.

Long before CFSG, Wielandt proved that any group not containing $A_n$ is at most $C\log n$-transitive, where $C \approx 1.998$ (see \cite{neumann}*{Section~2}; see \cite{dixon--mortimer}*{Theorem~5.5B} for a simplified proof with a slightly worse constant).
This appears to be where CFSG-free matters stand,
even after some 85 years.
The largest we can take $m$ in our proof is $(1 - o(1))\log_2 n$,
which is unfortunately a factor of about $1.4$ too small
to conflict with Wielandt's bound.

\subsection{Symmetric matrices}

It would be interesting to have variants of Theorems~\ref{main-four-prime-case} and \ref{main} in the case of random symmetric matrices,
including for example the adjacency matrix of a random graph.
It is known (first proved in \cite{CTV}) that a random symmetric matrix is nonsingular with high probability.
In the best-studied case of $\pm1$ entries, the strongest methods \citelist{\cite{ferber--jain}\cite{CM3}} prove this precisely by reducing mod $p$ for some prime $p \approx \exp(cn^{1/2})$
and showing that $M$ is nonsingular mod $p$ with high probability.%
\footnote{In a recent breakthrough, Campos, Jenssen, Michelen, and Sahasrabudhe have now proved that a random symmetric $\pm1$ matrix is singular with exponentially small probability~\cite{CJMS}.}
If one could prove that this probability is in fact close to $1-1/p$,
and if one could additionally allow an arbitrary shift in the diagonal entries (cf.~\cite{CTV}*{Section~6.4}),
then, assuming ERH, it would follow immediately by the \cite{BV} method that $\phi(t) = \det(t - M)$ is irreducible with high probability.%
\footnote{Following this suggestion, this has now been carried out by Ferber, Jain, Sah, Sawhney~\cite{FJSS}.}

\subsection{Notation}

Asymptotic notation is used incessantly, especially $O$ and $o$ notations.
Subscripted $O_\alpha$ warns that the implicit constant depends on $\alpha$.
We do not track explicitly the dependency on $\mu$, but it will not tax the motivated reader to work it out.
The notation $X \lesssim Y$ means $X \leq O(Y)$ (equivalent to number theorists' $\ll$), while $X \asymp Y$ means $X \lesssim Y$ and $X \gtrsim Y$.
Often we will simply write $C$ or $c$ to denote constants, $C$ suggesting a big constant and $c$ a little constant.
Such constants may change from line to line, though we will sometimes write $c'$ to warn that the constant has changed.

If $\mu$ is a probability measure we write $X \sim \mu$ to mean that $X$ is a random variable with law $\mu$.
We will occasionally write $X \sim S$ for a finite set $S$ to mean that $X$ is uniformly distributed on $S$, as in $M \sim \M_n(q)$, where $\M_n(q)$ denotes the set of $n\times n$ matrices over $\F_q$.

In the later sections of the paper we write $[n]$ for $\{1, \dots, n\}$.

\subsection*{Acknowledgements}

I am grateful to
P\'eter Varj\'u and Emmanuel Breuillard for advice and encouragement,
Hoi Nguyen for a discussion about Maples's method,
Kyle Luh for the proof of Proposition~\ref{one-simple-eig},
Bhargav Narayanan for discussions about the sparse problem in Section~\ref{sec:local-problem-part-2},
and Laci Pyber for comments on Wielandt's transitivity bound.

\section{The four prime method}

In this section we use the \cite{BSK} method to prove Theorem~\ref{main-four-prime-case},
which applies to any measure $\mu$ which is uniform modulo the product of some four distinct primes $N = p_1 p_2 p_3 p_4$,
such as the uniform measure on $\{1, \dots, 210\}$.
We will apply several results directly from \cite{BSK}.

The reduction $\phi \bmod p_i$ is the characteristic polynomial of a uniform random matrix in $\M_n(p_i)$,
and these reductions are independent for different $p_i$.
The factorization of $\phi \bmod p_i$ can be compared with the factorization of a random permutation $\pi \in S_n$ into cycles,
and by considering the Frobenius elements we can glean information about $\Gal(\phi)$.
The reason we need four primes is the following theorem about random permutations.

\begin{theorem}[\citelist{\cite{PPR-invariable} \cite{EFG-invariable}}]
  Say that $\pi_1, \dots, \pi_k \in S_n$ \emph{invariably generate at least $A_n$} if $\langle \pi_1', \dots, \pi_k'\rangle \geq A_n$ whenever $\pi_i'$ is conjugate to $\pi_i$ for each $i$.
  Then four random permutations invariably generate at least $A_n$ with probability bounded away from zero (but three do not).
\end{theorem}

The following lemma makes precise the sense in which the factorization of $\phi \bmod p_i$ can be compared with the factorization of a random permutation.

\begin{lemma}
  \label{TV}
  For $\Lambda$ a random partition of $n$ with some distribution, let $\Lambda_r$ denote the subsequence of parts of size at least $r$.
  For $\Lambda$ generated in any of the following ways the corresponding distributions $\Lambda_r$ are within $o(1)$ in total variation distance as $r,n \to \infty$:
  \begin{enumerate}[(1)]
    \item\label{distMn} the degrees of the irreducible factors of $\phi(t) = \det(t-M)$ where $M \sim \M_n(q)$, for any prime power $q$;
    \item\label{distGLn} the same but with $M \sim \GL_n(q)$;
    \item\label{distPn} the degrees of the irreducible factors of a uniformly random monic polynomial $\phi \in \F_q[t]$ of degree $n$, for any prime power $q$;
    \item\label{distSn} the cycle lengths of a random permutation $\pi \sim S_n$;
    \item\label{distZn} the partition consisting of $Z_i$ copies of $i$ for each $i \geq 1$, where $Z_1, Z_2, \dots$ are independent Poisson random variables with $\E Z_i = 1/i$, conditional on $\sum_{i \geq 1} i Z_i = n$.
  \end{enumerate}
\end{lemma}
\begin{proof}
  The equivalence \ref{distGLn} $\iff$ \ref{distPn} is proved in \cite{hansen-schmutz}.
  The equivalence of \ref{distPn} and \ref{distSn} is \cite{BSK}*{Lemma~6}.
  The distributions \ref{distSn} and \ref{distZn} are equivalent.
  The only link which does not seem to be in the literature is the least surprising one \ref{distMn} $\iff$ \ref{distGLn}.

  In any case, these and many other distributions are treated uniformly by \cite{ABT}.
  See particularly \cite{ABT}*{Theorem~3.2}, with $\theta=1$.
  Let $C_d$ be the number of parts of $\Lambda$ of size $d$.
  The only hypothesis necessary about $\Lambda$, which we must check, is that there are independent random variables $(Z_d : d \geq 1)$ satisfying the conditioning relation (CR)
  \[
    \P(C_1=c_1, \dots, C_n=c_n)
    = \P\br{Z_1=c_1, \dots, Z_n=c_n \;\middle|\; \sum_{d=1}^n d Z_d = n},
  \]
  and satisfying a technical condition called the uniform logarithmic condition (ULC),
  which asserts roughly that for large $d$ we may approximate $Z_d$ by a coin flip with expected value $1/d$: see \cite{ABT} for the precise formulation.

  We need not consider the distributions \ref{distSn} or \ref{distZn}, because \cite{ABT} uses \ref{distZn} as the reference distribution, and \ref{distSn} is equivalent.
  The distribution \ref{distPn} is covered by \cite{ABT}*{Proposition~1.1}.
  Hence it suffices to check (CR) and (ULC) for \ref{distMn} and \ref{distGLn}.

  Consider \ref{distMn}.
  From \cite{reiner}, the number of $g \in \M_n(q)$ with characteristic polynomial $f_1^{n_1} \cdots f_k^{n_k}$, where $f_i$ is irreducible and monic of degree $d_i$, is
  \[
    q^{n^2-n} \frac{F(q, n)}{\prod_{i=1}^k F(q^{d_i}, n_i)},
  \]
  where
  \[
    F(q, n) = q^{-n^2} |\GL_n(q)| = (1-q^{-1}) \cdots (1-q^{-n}).
  \]
  Let $C_d$ be the number of irreducible factors of $\phi(t) = \det(t - M)$ of degree $d$ counting multiplicity, where $M \sim \M_n(q)$.
  Let $I(d)$ be the number of irreducible polynomials in $\F_q[t]$ of degree $d$.
  Then $\P(C_1=c_1, \dots, C_n=c_n)$ is the coefficient of $u_1^{c_1} \cdots u_n^{c_n}$ in
  \[
    F(q,n)
    \prod_{d \geq 1} \br{\sum_{c \geq 0}
      \frac{u_d^c}{q^{dc} F(q^d, c)}
    }^{I(d)},
  \]
  assuming $\sum d c_d = n$.
  For each $d \geq 1$ let $Z_d$ be an independent random variable with probability generating function
  \begin{equation}
    \label{Mn-Zgf}
    \sum_{c \geq 0} \P(Z_d = c) u^c
    = \br{\zeta_d^{-1} \sum_{c \geq 0} \frac{u^c}{q^{dc} F(q^d, c)}}^{I(d)},
  \end{equation}
  where the normalizing factor $\zeta_d$ is
  \[
    \zeta_d = \sum_{c \geq 0} \frac1{q^{dc} F(q^d, c)}
    = 1 + q^{-d} + O(q^{-2d}).
  \]
  Then it follows that, if $\sum d c_d = n$,
  \begin{align*}
    \P(C_1=c_1, \dots, C_n=c_n)
     & = F(q,n) \prod_{d \geq1} \zeta_d^{I(d)} \P(Z_d = c_d) \\
     & = \beta_n \P(Z_1 = c_1, \dots, Z_n = c_n),
  \end{align*}
  for some constant $\beta_n$ independent of $c_1, \dots, c_n$.
  Hence (CR) is satisfied.
  Moreover, from the definition of $Z_d$, using $I(d) = q^d/d + O(q^{d/2})$,
  \begin{align*}
    \P(Z_d = 1)
     & = \zeta_d^{-I(d)} \frac{I(d)}{q^d F(q^d, 1)} \\
     & = e^{-1/d} / d + O(q^{-d/2})
  \end{align*}
  and
  \begin{align*}
    \P(Z_d = l)
     & \lesssim \zeta_d^{-I(d)} \binom{I(d) + l - 1}{l} q^{-dl} \\
     & \leq O(I(d)/q^d)^l                                       \\
     & \leq O(1/d)^l.
  \end{align*}
  Hence (ULC) is satisfied with, in the notation of \cite{ABT}, $e(d) = O(1/d)$ and $c_l = e^{-l}$.

  The case of \ref{distGLn} is almost identical, but the factor $f(t) = t$ must be excluded.
  Let $I'(d) = I(d)$ for $d > 1$ and let $I'(1) = I(1) - 1 = q-1$.
  Let $\phi(t) = \det(t - M)$ where $M \sim \GL_n(q)$ and let $C_d$ be the number of irreducible factors of degree $d$.
  Then $\P(C_1 = c_1, \dots, C_n = c_n)$ is the coefficient of $u_1^{c_1} \cdots u_n^{c_n}$ in
  \[
    \prod_{d \geq 1} \br{ \sum_{c \geq 0} \frac{u_d^c}{q^{dc} F(q^d, c)}}^{I'(d)},
  \]
  provided $\sum d c_d = n$.
  The rest of the verification is the same but with $I'(d)$ in place of $I(d)$, which makes no essential difference.

  Hence for $\Lambda, \Lambda'$ defined by any of \ref{distMn}--\ref{distZn},
  \cite{ABT}*{Theorem~3.2} shows that $\TV(\Lambda_r, \Lambda'_r) \to 0$ as $r, n \to \infty$.
\end{proof}

\begin{remark}
  The proof leads to a bound on the total variation distance of $O(1/r)$ (see \cite{ABT-book}*{Theorem~6.9}).
  A more special analysis likely demonstrates an exponential bound,
  but such a bound is not useful to us due to more severe losses in other parts of the argument.
\end{remark}

We now return to the setting of a random matrix $M \sim \M_n(\mu)$
and its characteristic polynomial $\phi(t) = \det(t - M)$.

\begin{proposition}
  \label{four-prime-prop-1}
  Assume there is a product of four distinct primes $N = p_1 p_2 p_3 p_4$ such that $\mu \bmod N$ is uniform.
  Let $E$ be the event that there is some $d \in [n^{1/4}, n)$ such that
  for each $i \in \{1, 2, 3, 4\}$ the reduction
  $\phi \bmod p_i$ has a divisor of degree $d$.
  Then $\P(E) = o(1)$.
\end{proposition}
\begin{proof}
  Let $\delta > 0$ be a small constant.
  Let $d_i$ be the sum of the degrees of all irreducible factors $\psi \mid \phi \bmod p_i$ of degree $\deg \psi < n^\delta$.
  Then obviously $n-d_i$ is the sum of the degrees of the irreducible factors of degree at least $n^\delta$.
  By the previous lemma with $r = n^\delta$, $d_i$ may be compared to the analogous quantity for a random permutation.
  Applying, e.g., \cite{BSK}*{Lemma~5}, it follows that $d_i \leq 2 n^{\delta}$ with high probability.

  Hence $E$ is almost contained in the event $E'$
  (meaning $\P(E' \setminus E) = o(1)$)
  that there is some $d \geq n^{1/4}$ such that for each $i$ the reduction $\phi \bmod p_i$ has a divisor of degree in $[d - 2n^\delta, d]$
  made up only of irreducible factors of degree at least $n^\delta$.
  Since this event depends only on the large irreducible factors we may again apply the previous lemma to compare it with the analogous event for a random permutation.
  Applying \cite{BSK}*{Lemma~8}, it follows that $E'$ is unlikely (provided $\delta$ is sufficiently small).
\end{proof}

Hence with high probability $\phi$ does not have a divisor of degree larger than $n^{1/4}$.
Small-degree factors can be ruled out more straightforwardly,
using the fact that there are not many possibilities for low-degree eigenvalues,
and the result \cite{bourgain-vu-wood}*{Corollary~3.3} that each fixed $\lambda$ has an exponentially small probability of being an eigenvalue.
For example, see \cite{O'RW}*{Theorem~2.4} for the case of $\pm1$ random variables.
Hence with high probability $\phi$ is irreducible.

Finally, if $\phi$ is irreducible, we claim that $\Gal(\phi) \geq A_n$ with high probability.
For this we need only assume that there is at least one prime $p$ such that $\mu \bmod p$ is uniform.

\begin{proposition}
  \label{four-prime-prop-2}
  Assume there is a prime $p$ such that $\mu \bmod p$ is uniform.
  Let $F$ be the event that $\Gal(\phi)$ is a transitive group other than $S_n$ or $A_n$.
  Then $\P(F) = o(1)$.
\end{proposition}
\begin{proof}
  Again let $\delta>0$ be a small constant.
  The reduction $\phi \bmod p$ is the characteristic polynomial of a random matrix $M \sim \M_n(p)$.
  Let $\Lambda$ be the partition of $n$ defined by the factorization of $\phi \bmod p$.
  Let $F'$ be the event that the parts of size smaller than $n^\delta$ can be adjusted so that $\Lambda$ becomes the cycle type of an element of some transitive group $G \ngeq A_n$.
  By Lemma~\ref{TV}, $F'$ may be compared with the analogous event for a random permutation.
  Hence, by \cite{BSK}*{Lemma~9}, $\P(F') = o(1)$.

  Moreover, the proof of \cite{BSK}*{Lemma~9} shows that $F \subset F'$.
  Briefly, the reduction $\phi \bmod p$ factorizes as
  \[
    \phi \bmod p = \psi_1 \psi_2 \qquad (\psi_1~\text{square-free}, \psi_2~\text{square-full}, \gcd(\psi_1, \psi_2) = 1),
  \]
  and with high probability $\deg \psi_2 < n^\delta$.
  Let $K$ be the splitting field of $\phi$, let $\pp$ be a prime over $p$, and let $k = K/\pp$ be the residue field, which is an extension of $\F_p$.
  The Frobenius automorphism of $k / \F_p$ lifts to an element of $\Gal(\phi) = \Gal(K/\Q)$,
  and its cycle type restricted to the roots of $\phi$ lying over roots of $\psi_1$ is determined by the factorization of $\psi_1$.
  Hence $F \subset F'$, so $\P(F) = o(1)$.
\end{proof}

\begin{remark}
  The bounds proved in this section are rather poor.
  The proof of Proposition~\ref{four-prime-prop-1} gives a bound of $n^{-c}$ for some constant $c>0$.
  The proof of Proposition~\ref{four-prime-prop-2} gives only $(\log n)^{-c}$, though this could likely be improved to $n^{-c}$ with further work.
\end{remark}

\section{Global to local via the prime ideal theorem}
\label{sec:local-global}

\def\Primes{P}  

In this section we reduce Theorem~\ref{main} to Theorems~\ref{thm:prob-M-nonsingular} and \ref{thm:correlation-bound}
using the \cite{BV} method.
Let $\Omega$ be the set of roots of $\phi$ in $\C$ and let $G = \Gal(\phi)$.
Let $R_\phi(p)$ be the number of roots of $\phi$ in $\F_p$.
It is a consequence of the prime ideal theorem
that $R_\phi(p)$ is on average close to the number of orbits $|\Omega / G|$;
in particular,
\[
  \phi~\text{irreducible (or a proper power)}
  \iff
  R_\phi(p) \sim 1 ~ \text{on average}.
\]
More generally for $m \geq 1$ the $m$th moment of $R_\phi(p)$ is $|\Omega^m/G| + o(1)$,
so
\[
  G~\text{is $m$-transitive}
  \iff
  R_\phi(p)^m \sim B_m ~ \text{on average},
\]
where $B_m$ is the $m$th Bell number.
Moreover, assuming ERH, there is a strong effective bound on how large $p$ must be for these asymptotics to hold.

Let $K$ be a number field.
Let $\Delta_K$ denote the discriminant of $K$.
If $\psi$ is a polynomial, let $\Delta_\psi$ denote the discriminant of $\psi$.
For each integer $q$ let $\Primes_K(q)$ be the number of prime ideals $\pp$ of $\OO_K$ of norm $q$.
For $\omega = (\omega_1, \dots, \omega_m) \in \Omega^m$,
let $\Q(\omega)$ denote the subfield of $\C$ generated by $\omega_1, \dots, \omega_m$.
Note that $\Q(\omega^g) = \Q(\omega)^g$ for any $g \in G$, so $\Primes_{\Q(\omega^g)}(p) = \Primes_{\Q(\omega)}(p)$,
so it makes sense to write $\Primes_{\Q(\omega)}(p)$ for $\omega \in \Omega^m/G$.

\begin{proposition}[\cite{BV}*{Proposition~16}]
  \label{BV-prop-16}
  Let $\phi \in \Z[x]$,
  let $\tilde \phi$ be the square-free part of $\phi$,
  and let $p$ be a prime not dividing $\Delta_{\tilde\phi}$.
  Then, for $m \geq 1$,
  \[
    R_\phi(p)^m = \sum_{\omega \in \Omega^m / G} \Primes_{\Q(\omega)}(p).
  \]
\end{proposition}

Rational primes will be weighted by $w_X(\log p)$, where
\[
  w_X(u) = u \cdot 2 \exp(-X) \indic{u \in (X - \log 2, X]}.
\]
Sums over all rational primes are denoted simply $\sum_p$.

\begin{proposition}[\cite{BV}*{Proposition~9}]
  \label{prime-ideal-theorem}
  If RH holds for $K$ then
  \[
    \sum_p \Primes_K(p) w_X(\log p) = 1 + O(X^2 \exp(-X/2) \log|\Delta_K|).
  \]
\end{proposition}

Combining the two propositions, we have the following one.

\begin{proposition}
  \label{local-global-general}
  Let $\phi \in \Z[x]$ be a polynomial of degree $n$ and let $m \geq 1$.
  Assume RH holds for $\Q(\omega)$ for each $\omega \in \Omega^m$.
  Then
  \[
    \sum_p R_\phi(p)^m w_X(\log p)
    = |\Omega^m/G|
    + O(m n^{2m-1} X^2 \exp(-X/2) \log |\Delta_{\tilde\phi}|).
  \]
\end{proposition}
\begin{proof}
  By the two previous propositions,
  \begin{align*}
    \sum_p R_\phi(p)^m w_X(\log p)
     & = \sum_p \sum_{\omega \in \Omega^m/G} \Primes_{\Q(\omega)}(p) w_X(\log p) + E          \\
     & = \sum_{\omega \in \Omega^m/G} (1 + O(X^2 \exp(-X/2) \log |\Delta_{\Q(\omega)}|)) + E,
  \end{align*}
  where $E$ is the error arising from primes $p \mid \Delta_{\tilde \phi}$.
  If $\psi \mid \tilde\phi$ is irreducible and $\omega_1$ is a root of $\psi$ then
  \[
    |\Delta_{\Q(\omega_1)}| \leq |\Delta_\psi| \leq |\Delta_{\tilde \phi}|
  \]
  (the quotient $|\Delta_\psi|/|\Delta_{\Q(\omega_1)}|$ is $[\OO_{\Q(\omega_1)} : \Z[\omega_1]]^2$).
  Recall that, for $K_1, K_2$ number fields,
  \[
    |\Delta_{K_1 K_2}| \leq |\Delta_{K_1}|^{[K_1K_2:K_1]} |\Delta_{K_2}|^{[K_1K_2:K_2]} \leq |\Delta_{K_1}|^{\deg K_2} |\Delta_{K_2}|^{\deg K_1}.
  \]
  Hence by induction we have, for $\omega \in \Omega^m$,
  \[
    |\Delta_{\Q(\omega)}| \leq |\Delta_{\tilde \phi}|^{m n^{m-1}}.
  \]
  Hence
  \begin{align*}
    \sum_p R_\phi(p)^m w_X(\log p)
     & = \sum_{\omega \in \Omega^m / G} (1 + O(X^2 \exp(-X/2) \cdot m n^{m-1} \log|\Delta_{\tilde \phi}|)) + E \\
     & = |\Omega^m / G| + O(m n^{2m-1} X^2 \exp(-X/2) \log|\Delta_{\tilde \phi}|) + E.
  \end{align*}
  We have $|R_\phi(p)| \leq n$ and $\sum_{\omega \in \Omega^m / G} \Primes_{\Q(\omega)}(p) \leq n^m$ identically,
  so each $p \mid \Delta_{\tilde \phi}$ contributes to $E$ at most
  \[
    n^m w_X(\log p) \lesssim n^m X \exp(-X)
  \]
  and there are at most $O(\log |\Delta_{\tilde \phi}|)$ such primes,
  so
  \[
    |E| \lesssim n^m X \exp(-X) \log |\Delta_{\tilde \phi}|.
  \]
  This error is subsumed by the other one.
\end{proof}

We need a bound for the discriminant of $\tilde \phi$ when $\phi$ is a characteristic polynomial.

\begin{lemma}
  Let $M$ be an $n \times n$ matrix whose entries are bounded by $H$
  and let $\phi$ be the characteristic polynomial of $M$.
  Let $\psi$ be any polynomial dividing $\phi$.
  Then
  \[
    |\Delta_\psi| \leq H^{n(n-1)} n^{n^2} .
  \]
\end{lemma}
\begin{proof}
  If $\lambda$ is an eigenvalue of $M$ then $|\lambda| \leq Hn$.
  Let $\lambda_1, \dots, \lambda_d$ be the roots of $\psi$.
  By definition $|\Delta_\psi|$ is the square of the determinant of the Vandermonde matrix $(\lambda_i^j)_{i,j}$.
  The $j$th column of this matrix has Euclidean norm bounded by $d^{1/2} (Hn)^j$.
  Thus by Hadamard's inequality
  \[
    |\Delta_\psi|^{1/2} = |\det(\lambda_i^j)|
    \leq d^{d/2} (Hn)^{d(d-1)/2}
    \leq n^{n/2} (Hn)^{n(n-1)/2}.
    \qedhere
  \]
\end{proof}

Combining the previous proposition and the lemma, we have:

\begin{proposition}
  \label{local-global}
  Assume $\phi$ is the characteristic polynomial of an $n \times n$ matrix with entries bounded by $H$.
  Let $m \geq 1$, and assume RH holds for all fields generated by at most $m$ roots of $\phi$.
  Then
  \[
    \sum_p R_\phi(p)^m w_X(\log p)
    = |\Omega^m/G|
    + O(m n^{2m+1} X^2 \exp(-X/2) \log(Hn)).
  \]
\end{proposition}

We have thus reduced the determination of $|\Omega^m / G|$ to the local problem of determining $R_\phi(p)^m$.
Note that, if $\phi(t) = \det(t-M)$ then $R_\phi(p)$ is simply the number of eigenvalues of $M$ acting on the finite vector space $\F_p^n$.
We turn to this problem in the next two sections, in which we will prove Theorems~\ref{thm:prob-M-nonsingular} and \ref{thm:correlation-bound}.
Assume for now those theorems have been proved and we will prove Theorem~\ref{main}.

Consider first $m = 1$.
Assume $\mu$ is supported on $[-H,H]$.
For $\lambda \in \F_p$, let $E_\lambda$ be the event that $\phi(\lambda)=0$, i.e., that $\lambda$ is an eigenvalue of $M$ mod $p$.
For $p > 2H+1$, $\mu \bmod p$ is $\alpha$-balanced for some constant $\alpha$ (since $\mu$ is nontrivial).
By Theorem~\ref{thm:prob-M-nonsingular} applied to $M - \lambda$,
\[
  \P(E_\lambda) = 1/p + O(1/p^2 + e^{-cn}).
\]
Hence, by linearity of expectation,
\[
  \E R_\phi(p) = 1 + O(1/p + p e^{-cn}).
\]
It follows that
\begin{align*}
  \E \sum_p R_\phi(p) w_X(\log p)
   & = \sum_p (1 + O(1/p + p e^{-c n})) w_X(\log p)      \\
   & = \sum_p (1 + O(e^{-X} + e^{X - c n})) w_X(\log p).
\end{align*}
Applying Proposition~\ref{prime-ideal-theorem} with $K = \Q$,
this is
\[
  (1 + O(e^{-X} + e^{X - c n}))
  (1 + O(X^2 e^{-X/2}))
  = 1 + O(X^2 e^{-X/2} + e^{X - c n})
  .
\]
On the other hand by Proposition~\ref{local-global} it is also
\[
  \E |\Omega / G| + O(n^3 X^2 \exp(-X/2) \log(Hn)).
\]
Taking $X = c n / 2$, we deduce that
\[
  \E |\Omega / G| = 1 + O(n^5 \log(Hn) e^{-c' n}) = 1 + O(e^{-c''n}).
\]
Hence
\[
  \P(|\Omega / G| > 1) \lesssim e^{-c' n}.
\]
Hence $\phi$ has a unique irreducible factor with probability $1 - O(e^{-cn})$.

The possibility that $\phi$ is a proper power
is ruled out by the following proposition,
which shows that $\phi$ has at least one simple root in $\C$.
(Of course, if $\phi$ is irreducible then the entire spectrum of $M$ is simple.)

\begin{proposition}
  \label{one-simple-eig}
  With probability $1 - O(e^{-cn})$,
  the matrix $M$ has at least one simple complex eigenvalue.
\end{proposition}
\begin{proof}
  Let $\bar\mu$ be the mean of $\mu$.
  If $\bar\mu = 0$ then the entire spectrum of $M$ is simple with probability $1 - O(e^{-cn})$ by \cite{l-o'r-simple-spec}*{Corollary~1.10}.
  Suppose $\bar\mu \neq 0$. Let $N = M - \bar\mu J$, where $J$ is the ones matrix.
  Then $\|N\| \lesssim n^{1/2} $ with probability $1 - O(e^{-cn})$ (see \cite{vershynin-book}*{Theorem~4.4.5}),
  while the spectrum of $\bar\mu J$ consists of $\bar\mu n$ and $0$ with multiplicity $n-1$,
  so a continuity argument as in \cite{silverstein}*{p.~526} shows that $M$ has a unique eigenvalue $\lambda$ such that $|\lambda - \bar\mu n| \leq \|N\|$.
  In detail, let $M(t) = \bar\mu J + N t$ for $t \in [0, 1]$, so $M(0) = \bar\mu J$ and $M(1) = M$.
  By \cite{horn-johnson}*{Corollary~6.3.4}, every eigenvalue of $M(t)$ is within $\|N\|$ of $0$ or $\bar \mu n$.
  These discs are disjoint for large enough $n$ (depending on $\bar \mu$), and the eigenvalues of $M(t)$ vary continuously,
  so there can be only one eigenvalue of $M(t)$ within $\|N\|$ of $\bar \mu n$ for all $t \in [0, 1]$
  (and it must remain real).%
  \footnote{The bound provided by this argument is unfortunately not uniform in $\bar\mu$ near 0, so neither is our main theorem.}
\end{proof}

Now consider arbitrary $m \geq 2$.
Assume $m < c \log n$ and $C \leq \log p < c n^{1/2} / m^{3/2}$.
By Theorem~\ref{thm:correlation-bound}, for distinct $\lambda_1, \dots, \lambda_m \in \F_p$,
\[
  \P(E_{\lambda_1} \cap \cdots \cap E_{\lambda_m})
  \leq (p-1)^{-m} + O\br{\exp(-c n^{1/2} / m^{3/2})}.
\]
It follows that
\begin{align*}
  \sum_{\lambda_1, \dots, \lambda_m~\text{distinct}} \P(E_{\lambda_1} \cap \cdots \cap E_{\lambda_m})
   & \leq
  \frac{p(p-1) \cdots (p-m+1)}{(p-1)^m} + O(p^m e^{-cn^{1/2} / m^{3/2}})
  \\
   & \leq
  1 + O(1/p + p^m e^{-cn^{1/2} / m^{3/2}})
  .
\end{align*}
Summing over all $\lambda_1, \dots, \lambda_m$ and considering all possible partitions of $\{1, \dots, m\}$ defined by equality among the $\lambda$'s,
it follows that
\[
  \E R_\phi(p)^m
  \leq
  B_m
  (1 + O(1/p + p^m e^{-cn^{1/2} / m^{3/2}})).
\]
Averaging over primes, and again using Proposition~\ref{prime-ideal-theorem} with $K = \Q$,
\begin{align*}
  \E \sum_p R_\phi(p)^m w_X(\log p)
   & \leq
  \sum_p B_m (1 + O(1 / p + p^m e^{-c n^{1/2} / m^{3/2}})) w_X(\log p)  \\
   & =
  B_m
  (1 + O(e^{-X} + e^{m X - c n^{1/2} / m^{3/2}})) (1 + O(X^2 e^{-X/2})) \\
   & =
  B_m
  \br{1 + O(X^2 e^{-X/2} + e^{m X - c n^{1/2} / m^{3/2}})}.
\end{align*}
Fix $X = cn^{1/2} / m^{5/2}$ for a sufficiently small constant $c$.
Using $B_m \leq m^m$, it follows that
\[
  \E \sum_p R_\phi(p)^m w_X(\log p)
  \leq B_m +
  O\br{
  e^{-cn^{1/2} / m^{5/2}}
  }.
\]
On the other hand, by Proposition~\ref{local-global} we also have
\[
  \E \sum_p R_\phi(p)^m w_X(\log p)
  = \E |\Omega^m / G| + O(m n^{2m+1} X^2 e^{-X/2} \log(Hn)).
\]
Hence
\[
  \E|\Omega^m / G|
  \leq
  B_m
  + O\br{
  e^{-cn^{1/2} / m^{5/2}}
  }.
\]
Trivially $|\Omega^m / G| \geq B_m$, so
\[
  \P(|\Omega^m / G| > B_m)
  \lesssim
  e^{-cn^{1/2} / m^{5/2}}
  .
\]
Note that $|\Omega^m / G| = B_m$ if and only if $G$ is $m$-transitive.
Taking $m$ to be a sufficiently large constant,
it follows from Fact~\ref{6trans} that $\Gal(\phi) \geq A_n$ with probability $1 - O(e^{-cn^{1/2}})$, as claimed.

\section{The local problem, part 1: singularity of matrices over \texorpdfstring{$\F_q$}{F\_q}}

In this section we prove Theorem~\ref{thm:prob-M-nonsingular}.
Recall the context:
A measure $\mu$ on $\F_q$ is called \emph{$\alpha$-balanced} if
$\mu(x+H) \leq 1 - \alpha$ for every $x \in \F_q$ and every proper subgroup $H < \F_q$.
Let $\bmu = (\mu_{ij})_{1 \leq i,j\leq n}$ be a matrix of $\alpha$-balanced measures on $\F_q$, and let $\M(\bmu)$ denote the distribution of matrices $M \in \M_n(q)$ with independent entries and $M_{ij} \sim \mu_{ij}$.
The claim is that
\[
  \P(M~\textup{nonsingular}) = \prod_{i=1}^\infty (1-1/q^i) + O_\alpha(e^{-c\alpha n}).
\]

Fix some constant parameters
\[
  \zeta \ll \delta \ll \eta \ll 1.
\]
Let $X_1, \dots, X_n$ be the rows of $M$, so
\begin{equation}
  \label{Xi-law}
  X_i \sim \mu_{i1} \otimes \cdots \otimes \mu_{in}.
\end{equation}
Let $M_k$ be the top $k \times n$ submatrix.
Let $\row (M_k)$ be the row space of $M_k$, i.e.,
\[
  \row (M_k) = \langle X_1, \dots, X_k\rangle.
\]
We will prove that $\row (M_k)$ is suitably generic with respect to $\bmu$ with high probability
for all $k \geq (1-\eta)n$.
For a subspace $V \leq \F_q^n$, let
\begin{align*}
  \rho_i(V) & = \max_{x \in \F_q^n} |\P(X_i \in x + V) - 1/q^{\codim V}|; \\
  \rho(V)   & = \max_i \rho_i(V).
\end{align*}
The quantity $\rho_i(V)$ measures the nonuniformity of $X_i$ mod $V$,
and $\rho(V)$ measures the maximum nonuniformity among $X_1, \dots, X_n$ mod $V$.
For $v \in \F_q^n$, the \emph{support} of $v$ is
\[
  \supp v = \{ i \in [n] : v_i \neq 0 \}.
\]
We use the following taxonomy for subspaces $V \leq \F_q^n$
(adapted from \cite{maples}):
\begin{description}
  \item[sparse] $V$ is called \emph{sparse} if there is a nonzero $v \perp V$ with $|\supp v| \leq \delta n$;
  \item[unsaturated] $V$ is called \emph{unsaturated} if $V$ is not sparse and
        \[
          \rho(V) > \max(e^{-\zeta \alpha n}, 10/q^{\codim V});
        \]
  \item[semi-saturated] $V$ is called \emph{semi-saturated} if $V$ is not sparse and
        \[
          e^{-\zeta \alpha n} < \rho(V) \leq 10/q^{\codim V};
        \]
  \item[saturated] $V$ is called \emph{saturated} if $V$ is not sparse and
        \[
          \rho(V) \leq e^{-\zeta \alpha n}.
        \]
\end{description}

In this language, the main assertion is the following.

\begin{theorem}
  \label{thm:saturation}
  For $M \sim \M(\bmu)$ and $k \geq (1- \eta)n$,
  \[
    \P(\row (M_k)~\textup{is saturated}) = 1 - O_\alpha(e^{-c\alpha n}).
  \]
\end{theorem}

The proof occupies the rest of this section, but first let us see how Theorem~\ref{thm:prob-M-nonsingular} follows.
The following theorem is a generalization.

\begin{theorem}
  \label{thm:prob-Mk-nonsingular}
  For $M \sim \M(\bmu)$,
  \[
    \P(\dim \row (M_k) = k) = \prod_{i=n-k+1}^\infty (1-1/q^i) + O_\alpha(e^{-c\alpha n}).
  \]
\end{theorem}

To deal with $k < (1-\eta)n$ we will use the following well-known lemma of Odlyzko,
which will also be used repeatedly in the rest of the paper.

\begin{lemma}[Odlyzko, see \citelist{\cite{bourgain-vu-wood}*{Lemma~B.1}
  \cite{maples}*{Lemma~2.2}}]
  \label{odlyzko}
  If $V \leq \F_q^n$ is an affine subspace and $X \in \F_q^n$ has independent $\alpha$-balanced entries then
  \[
    \P(X \in V) \leq (1-\alpha)^{\codim V}.
  \]
\end{lemma}

\begin{proof}[Proof of Theorem~\ref{thm:prob-Mk-nonsingular}]
  Let
  \[
    \eps_k = \left|
    \P(\dim \row (M_k) = k)
    - \prod_{i = n-k+1}^\infty (1-1/q^i)
    \right|.
  \]
  We claim that $\eps_k \leq e^{-c\alpha n}$ for all $k$.
  If $k-1 < (1-\eta)n$, then by Odlyzko's lemma
  \[
    \P(\dim \row (M_k) < k)
    \leq k (1 - \alpha)^{n-k+1} \leq n e^{-\alpha \eta n},
  \]
  and
  \[
    \prod_{i=n-k+1}^\infty (1-1/q^i) = 1 + O(1/q^{n-k+1}) = 1 + O(2^{-\eta n}),
  \]
  so the claim holds trivially.
  Suppose $k-1 \geq (1-\eta) n$.
  We have $\dim \row (M_k) = k$ if and only if $\dim \row (M_{k-1}) = k-1$ and $X_k \notin \row (M_{k-1})$.
  Conditional on $\row (M_{k-1})$ being saturated and $(k-1)$-dimensional,
  \[
    \P(X_k \notin \row (M_{k-1})) = 1-1/q^{n-k+1} + O(e^{-c\alpha n}).
  \]
  Hence, using Theorem~\ref{thm:saturation},
  \begin{align*}
    \P(\dim \row (M_k) = k)
     & = \P(\dim \row (M_{k-1}) = k-1) (1-1/q^{n-k+1}+ O(e^{-c\alpha n}))
    \\
     & \qquad + O(\P(\row (M_{k-1})~\text{not saturated}))                         \\
     & = \P(\dim \row (M_{k-1}) = k-1) (1-1/q^{n-k+1}) + O_\alpha(e^{-c\alpha n}).
  \end{align*}
  Hence
  \[
    \eps_k \leq \eps_{k-1} + O_\alpha(e^{-c\alpha n}).
  \]
  It follows that $\eps_k \lesssim_\alpha n e^{-c\alpha n} \lesssim_\alpha e^{-c'\alpha n}$, as claimed.
\end{proof}

We now proceed with the proof of Theorem~\ref{thm:saturation}.
The proof consists of bounding the probability that $\row (M_k)$ is sparse, semi-saturated, or unsaturated.
The unsaturated case is the most interesting and will be handled last.

\subsection{Sparse subspaces}

Note $\row (M_k)^\perp = \ker M_k$.
Suppose $\supp v = \{j_1, \dots, j_s\}$.
Let $C_j \in \F_q^k$ denote the $j$th column of $M_k$.
If $M_k v = 0$ then
\begin{equation}
  \label{Cjs}
  C_{j_s} \in \langle C_{j_1}, \dots, C_{j_{s-1}} \rangle.
\end{equation}
For fixed $j_1, \dots, j_s$, \eqref{Cjs} occurs with probability at most $(1-\alpha)^{k-s+1}$ by Odlyzko's lemma.
Hence, by the union bound,
\[
  \P(\row (M_k)~\text{is sparse})
  \leq \sum_{s \leq \delta n} \binom{n}{s} (1-\alpha)^{k-s+1}.
\]
Assuming $\delta = \delta(\alpha)$ is sufficiently small and $k \geq (1-\eta)n$, this is bounded by $e^{-c\alpha n}$ as desired.

\begin{remark}
  To make $\delta$ and ultimately the constant $c$ in Theorem~\ref{thm:saturation} independent of $\alpha$, use the above argument only for $s$ up to $O(\alpha^{-1})$.
  For larger $s$, use the following argument.

  Suppose $v \in \ker M_k$ has minimal support, say $S \subset [n]$ of size $s$.
  The set of vectors $u \in \ker M_k$ supported on $S$ must be precisely the line $\langle v \rangle$, or else we could find a linear combination with smaller support.
  There are some $s-1$ rows $(X_i : i \in T)$ of $M_k$ whose restrictions to $S$ span $v^\perp|_S$,
  and in particular determine $v$ up to a scalar.
  For fixed $S$, $T$, and $v$, and $M' = (X_i : i \notin T)$,
  \[
    \P(M' v = 0) = \prod_{i \notin T} \P(X_i \cdot v = 0).
  \]
  By \cite{maples}*{Lemma~2.4} or \cite{nguyen-paquette}*{Theorem~A.21}, $\P(M' v = 0) \leq (2/3)^{k-s+1}$.
  Hence by summing over $S$ and $T$ we get the bound
  \[
    \sum_{C\alpha^{-1} < s \leq \delta n} \binom{n}{s} \binom{k}{s-1} (2/3)^{k-s+1},
  \]
  which is negligible as long as $\delta$ is sufficiently small independently of $\alpha$.

  This improvement is not important for the application to Theorem~\ref{main}, in which $\alpha$ is considered a constant.
\end{remark}

\subsection{Semi-saturated subspaces}

The \emph{large spectrum} of a measure $\mu$ on $\F_q$ is
\[
  \Spec_t \mu = \{ u \in \F_q : |\muhat(u)| \geq t \}.
\]
Here we are identifying $\F_q$ with its own dual group in the usual way by fixing a nontrivial character $\chi : \F_q \to S^1$ and defining
\[
  \muhat(u) = \sum_{x \in \F_q} \mu(x) \chi(-ux).
\]

\begin{lemma}[\cite{maples}*{Lemma~3.2}; see also \cite{tao-vu}*{Lemma~4.37}]
  \label{spec-lemma}
  For $\eps_1, \dots, \eps_k \geq 0$,
  \[
    \Spec_{1-\eps_1} \mu + \cdots + \Spec_{1-\eps_k} \mu
    \subset \Spec_{1-k(\eps_1 + \cdots + \eps_k)} \mu.
  \]
\end{lemma}

The following lemma is essentially contained in \cite{maples}*{Section~3.2}.

\begin{lemma}
  \label{small-spec}
  Suppose $\mu$ is an $\alpha$-balanced measure on $\F_q$.
  Then
  \[
    |\Spec_{1-\eps \alpha} \mu \minuso| \lesssim \eps^{1/2} q.
  \]
\end{lemma}
\begin{proof}
  Since $\mu$ is $\alpha$-balanced, for any nonzero subgroup $H \leq \F_q$ we have
  \[
    \min_H |\muhat|^2
    \leq \|\muhat|_H\|_2^2
    = \|\mu \bmod H^\perp\|_2^2
    \leq 1-\alpha.
  \]
  Hence $\Spec_{1-\alpha/2}\mu$ does not contain a subgroup.
  On the other hand, by Lemma~\ref{spec-lemma},
  \[
    k \Spec_{1-k^{-2} \alpha / 2} \mu \subset \Spec_{1 - \alpha/2} \mu.
  \]
  Hence by Kneser's theorem
  \[
    k |\Spec_{1-k^{-2} \alpha / 2} \mu \minuso| \leq |\Spec_{1-\alpha/2} \mu\minuso| < q.
  \]
  Take $k \asymp \eps^{-1/2}$.
\end{proof}

Now suppose $V$ is a semi-saturated subspace of codimension $d \geq n-k$, so
\begin{equation}
  \label{semi-sat}
  e^{-\zeta \alpha n} < \max_i \rho_i(V) \leq 10 / q^d .
\end{equation}
Let $i$ be an index realizing the maximum.
By Fourier analysis,
\begin{align*}
  \P(X_i \in x + V)
   & = \frac1{|V^\perp|} \sum_{v \in V^\perp} \chi(x \cdot v) \prod_{j=1}^n \muhat_{ij}(v_j),
  \\
  \rho_i(V)
   & \leq \frac1{|V^\perp|} \sum_{v \in V^\perp \minuso} \prod_{j=1}^n |\muhat_{ij}(v_j)|.
\end{align*}
Hence there is some $v \in V^\perp \minuso$ such that
\[
  \rho_i(V) \leq \prod_{j=1}^n |\muhat_{ij}(v_j)|.
\]
Taking logarithms and using ${\log x \leq x - 1}$,
\[
  \sum_{j=1}^n (1 - |\muhat_{ij}(v_j)|)
  \leq \log \rho_i(V)^{-1}
  < \zeta \alpha n.
\]
By Markov's inequality there is a set $J \subset [n]$ of size $|J| \geq (1-\delta/2)n$ such that
\[
  1 - |\muhat_{ij}(v_j)| \leq 2 \delta^{-1} \zeta \alpha
\]
for all $j \in J$, i.e.,
\[
  v_j \in \Spec_{1 - 2 \delta^{-1} \zeta \alpha} \mu_{ij}.
\]
Moreover, since $V$ is not sparse, $|\supp v| > \delta n$,
so there is a set $J' = J \cap \supp v$ of size $|J'| \geq (\delta/2)n$
such that
\[
  v_j \in \Spec_{1 - 2 \delta^{-1} \zeta \alpha} \mu_{ij} \minuso
\]
for all $j \in J'$.
By Lemma~\ref{small-spec},
\[
  |\Spec_{1-2\delta^{-1} \zeta \alpha} \mu_{ij} \minuso |
  \lesssim \delta^{-1/2} \zeta^{1/2} q.
\]
Hence the number of possibilities for $v$,
accounting for the choice of $i \in [n]$ and $J' \subset [n]$,
is bounded by
\[
  n 2^n (C \delta^{-1/2} \zeta^{1/2} q)^{(\delta/2)n} q^{(1-\delta/2)n}
  = O(1)^n (\zeta / \delta)^{\delta n/4} q^n.
\]
The number of $d$-dimensional subspaces containing $v$ is $O(q^{n(d-1) - d(d-1)})$,
so the number of possibilities for $V$ is
\begin{equation}
  \label{semi-sat-poss}
  O(1)^n (\zeta / \delta)^{\delta n/4} q^{dn - d(d-1)}.
\end{equation}
For any fixed such $V$, we have
\[
  \P(\row (M_k) \leq V) \leq (11 / q^d)^k,
\]
by \eqref{semi-sat}.
Hence
\begin{align*}
  \P(\row (M_k) \leq \textup{some semi-saturated}~ & V~\text{of codim}~d)
  \\
                                                 & \lesssim O(1)^n (\zeta / \delta)^{\delta n/4} q^{dn - d(d-1)} (11 / q^d)^k
  \\
                                                 & = O(1)^n (\zeta / \delta)^{\delta n/4} q^{d(n - d + 1 - k)}
  \\
                                                 & \leq O(1)^n (\zeta/\delta)^{\delta n/4} q^d,
\end{align*}
where in the last line we used $d \geq n-k$.
Using \eqref{semi-sat} again, this is bounded by
\[
  O(1)^n (\zeta / \delta)^{\delta n/4} e^{\zeta \alpha n}.
\]
As long as $\zeta$ is sufficiently small depending on $\delta$ this is exponentially negligible.

\subsection{Unsaturated subspaces}

Finally, we consider unsaturated subspaces.
Let $\UU_i$ be the set of unsaturated subspaces $V \leq \F_q^n$ such that
\[
  \rho(V) = \rho_i(V).
\]
It suffices to show that
\begin{equation}
  \label{unsat-goal}
  \P(\row (M_k) \in \UU_i) \leq e^{-cn}
\end{equation}
for each $i$ separately.
To do this we will construct
new measures $(\nu_{ij})_{1 \leq j \leq n}$,
each still $\alpha/8$-balanced,
such that if (cf.~\eqref{Xi-law})
\begin{equation}
  \label{Yi-law}
  Y \sim \nu_{i1} \otimes \cdots \otimes \nu_{in}
\end{equation}
and $N$ is $M_k$ but with most of the rows replaced with independent copies of $Y$
then, for any $V \in \UU_i$,
\[
  \P(\row (M_k) = V) \leq e^{-cn} \P(\row (N) = V).
\]
We will then use the disjointness of the events $\{\row (N) = V\}$ to infer \eqref{unsat-goal}.

Let $Y_1, \dots, Y_r$ be independent copies of $Y$,
and $V$ a subspace.
Let
\[
  B_V = \{Y_1, \dots, Y_r~\text{lin.~ind.~and in}~V\}.
\]
Clearly
\begin{align*}
  \P(B_V)
   & =
  \P(Y_1, \dots, Y_r \in V)
  - \P(Y_1, \dots, Y_r ~ \text{lin.~dep.~in}~V) \\
   & =
  \P(Y \in V)^r
  - \P(Y_1, \dots, Y_r ~ \text{lin.~dep.~in}~V),
\end{align*}
and by Odlyzko's lemma
\[
  \P(Y_1, \dots, Y_r ~ \text{lin.~dep.~in}~V)
  \leq r \P(Y \in V)^{r-1} (1-\alpha/8)^{n-r+1}
\]
so
\begin{equation}
  \label{BV-ratio}
  \frac{\P(B_V)}{\P(Y \in V)^r}
  \geq 1 - \frac{r (1 - \alpha/8)^{n-r+1}}{ \P(Y \in V)}.
\end{equation}
When we fix parameters we will ensure that this is at least $1/2$.

By independence of $X_1, \dots, X_k$ and $Y_1, \dots, Y_r$,
\[
  \P(\row (M_k) = V)
  = \frac{\P(B_V \wedge \row (M_k) = V)}{\P(B_V)}.
\]
If $B_V$ holds and $\row (M_k) = V$ then,
by the Steinitz exchange lemma from elementary linear algebra,
we can find $k - r$ rows of $M_k$ that together with $Y_1, \dots, Y_r$ span $V$,
and the remaining rows of $M_k$ must also be contained in $V$;
hence
\[
  \P(B_V \wedge \row (M_k) = V)
  \leq \sum_{\substack{R \subset [k] \\ |R|=r}}
  \P(\row (N_R) = V) \prod_{i' \in R} \P(X_{i'} \in V),
\]
where $N_R$ is $M_k$ but with the rows $(X_{i'} : i' \in R)$ replaced with $(Y_1, \dots, Y_r)$.
Thus
\[
  \P(\row (M_k) = V)
  \lesssim \sum_{\substack{R \subset [k] \\ |R|=r}}
  \P(\row (N_R) = V) \max_{i'} \pfrac{\P(X_{i'} \in V)}{\P(Y \in V)}^r.
\]
Summing over $V \in \UU_i$,
\begin{align*}
  \P(\row (M_k) \in \UU_i)
   & \lesssim \sum_{\substack{R \subset [k]         \\ |R|=r}}
  \P(\row (N_R) \in \UU_i) \max_{V \in \UU_i}
  \max_{i'} \pfrac{\P(X_{i'} \in V)}{\P(Y \in V)}^r \\
   & \leq \binom{k}{r} \max_{V \in \UU_i}
  \max_{i'} \pfrac{\P(X_{i'} \in V)}{\P(Y \in V)}^r.
\end{align*}
Fix $r = \floor{0.99k}$.
To complete the proof of \eqref{unsat-goal},
we must show that $(\nu_{ij})_{1\leq j\leq n}$ can be chosen so that, for all $V \in \UU_i$,
\begin{enumerate}[(a)]
  \item\label{a} $\P(Y \in V) > 2r (1-\alpha/8)^{n-r+1}$ (so that \eqref{BV-ratio} $\geq 1/2$);
  \item\label{b} $\max_{i'} \P(X_{i'} \in V) / \P(Y \in V) \leq 0.6$.
\end{enumerate}

We will construct $(\nu_{ij})_{1 \leq j \leq n}$ so that, for all $V \in \UU_i$,
\begin{equation}
  \label{eqn-to-arrange}
  \rho_i(V) \leq (1/2 + o(1)) |\P(Y \in V) - 1/q^{\codim V}|,
\end{equation}
and the conditions \ref{a} and \ref{b} above follow from this.
Indeed, for $V \in \UU_i$ of codimension $d$,
\[
  |\P(Y \in V) - 1/q^d| \geq (2 + o(1)) \max(e^{-\zeta \alpha n}, 10/q^d).
\]
Since the right-hand side is bigger than $1/q^d$
the absolute values on the left-hand side may be dropped,
so certainly \ref{a} is satisfied (since $r \leq 0.99n$ and $\zeta \ll 1$).
Also
\begin{align*}
  \P(X_{i'} \in V)
   & \leq 1/q^d + \rho(V)                             \\
   & \leq 1/q^d + (1/2 + o(1)) (\P(Y \in V) - 1/q^d)) \\
   & = (1/2 + o(1)) (\P(Y \in V) + 1/q^d)             \\
   & \leq (1/2 + 1/21 + o(1)) \P(Y \in V),
\end{align*}
so \ref{b} is satisfied too.

\begin{lemma}[Cf.~\cite{maples}*{Proposition~3.6}]
  \label{lem:nu}
  Let $\mu$ be an $\alpha$-balanced measure on $\F_q$.
  There is a probability measure $\nu$ with the following properties:
  \begin{enumerate}[(1)]
    \item\label{nu1} $\nu$ is $\alpha/8$-balanced;
    \item\label{nu2} $\nuhat > 0$;
    \item\label{nu3} $\nuhat^4 \geq |\muhat|$;
    \item\label{nu4} $\nuhat(s + t)^2 \geq |\muhat(s)| |\muhat(t)|$ for all $s, t \in \F_q$.
  \end{enumerate}
\end{lemma}
\begin{proof}
  Fix any constant $\gamma \leq 1/8$ and let
  \[
    \nu = (1-\gamma)\delta + \gamma \mu * \mu^-,
  \]
  where $\delta$ is the mass at zero
  and $\mu^-$ is the image of $\mu$ under $x \mapsto -x$.
  Clearly $\nu$ is $\gamma\alpha$-balanced, so \ref{nu1} holds.
  The Fourier transform of $\nu$ is
  \[
    \nuhat = 1-\gamma + \gamma |\muhat|^2,
  \]
  so \ref{nu2} holds.
  Suppose $|\muhat(s)| = 1-\eps$.
  Then
  \[
    \nuhat(s)
    = 1 - \gamma + \gamma(1 - \eps)^2
    \geq 1 - 2\gamma\eps,
  \]
  so
  \[
    \nuhat(s)^4 \geq 1 - 8\gamma\eps.
  \]
  Since $\gamma \leq 1/8$ this proves \ref{nu3}.
  For \ref{nu4}, let $\eps_1 = 1 - |\muhat(s)|$ and $\eps_2 = 1 - |\muhat(t)|$.
  Then, by Lemma~\ref{spec-lemma}, $|\muhat(s+t)| \geq 1 - 2(\eps_1+\eps_2)$, so
  \begin{align*}
    \nuhat(s+t)
     & \geq 1 - \gamma + \gamma (1 - 4(\eps_1 + \eps_2))        \\
     & = 1 - 4\gamma(\eps_1+\eps_2)                             \\
     & \geq (1 - 8\gamma\eps_1)^{1/2} (1 - 8\gamma\eps_2)^{1/2}
    .
  \end{align*}
  Since $\gamma \leq 1/8$ this proves \ref{nu4}.
\end{proof}

Let $\nu_{ij}$ be the measure constructed as above from $\mu_{ij}$.
Let $Y$ satisfy \eqref{Yi-law}. The following lemma verifies \eqref{eqn-to-arrange} and thus completes the proof of \eqref{unsat-goal}.

\begin{lemma}[Cf.~\cite{maples}*{Lemma~2.8}]
  \label{Y-lemma}
  Let $V$ be a subspace, and assume every nonzero $v \perp V$ has at least $s$ nonzero entries. Then
  \[
    \rho_i(V) \leq \br{1/2 + e^{-c\alpha s}} |\P(Y \in V) - 1/q^{\codim V}|.
  \]
\end{lemma}
\begin{proof}
  Define
  \begin{align*}
    f(v)  & = \prod_{j=1}^n \muhat_{ij}(v_j),
          &
    F(t)  & = \{v \in V^\perp : |f(v)| \geq t\},
          &
    F'(t) & = F(t) \minuso,
    \\
    g(v)  & = \prod_{j=1}^n \nuhat_{ij}(v_j),
          &
    G(t)  & = \{v \in V^\perp : g(v) \geq t\},
          &
    G'(t) & = G(t) \minuso.
  \end{align*}
  By Fourier analysis,
  \begin{align*}
    \P(X_i \in x + V) & = \frac1{|V^\perp|} \sum_{v \in V^\perp} \chi(x \cdot v) f(v), \\
    \P(Y \in V)       & = \frac1{|V^\perp|} \sum_{v \in V^\perp} g(v),
  \end{align*}
  so it suffices to prove
  \[
    \sum_{v \in V^\perp \minuso} |f(v)|
    \leq \br{\frac12 + e^{-c\alpha s}} \sum_{v \in V^\perp \minuso} g(v).
  \]

  Since $|f(v)| \leq g(v)^4$ (by Lemma~\ref{lem:nu}\ref{nu3}) we have
  \[
    \sum_{v \in V^\perp \minuso : f(v) < \eps} |f(v)|
    \leq \eps^{3/4} \sum_{v \in V^\perp \minuso} g(v).
  \]
  The other part is
  \[
    \sum_{v \in V^\perp \minuso : f(v) \geq \eps} |f(v)|
    = \int_\eps^\infty |F'(t)| \, dt + \eps |F'(\eps)|.
  \]
  By Lemma~\ref{lem:nu}\ref{nu4} and tensorization,
  \[
    g(u+v)^2 \geq |f(u)| |f(v)|
  \]
  for all $u, v \in \F_q^n$. Hence
  \[
    F(t) + F(t) \subset G(t).
  \]
  Thus by Kneser's theorem either
  \begin{equation}
    \label{2F'}
    2|F'(t)| \leq |G'(t)|
  \end{equation}
  or $G(t)$ contains a nontrivial subgroup.
  Assume $G(\eps)$ does not contain a nontrivial subgroup.
  Then \eqref{2F'} holds for all $t \geq \eps$, so
  \[
    \int_\eps^\infty |F'(t)| \, dt + \eps |F'(\eps)|
    \leq \frac12 \br{
      \int_\eps^\infty |G'(t)| \, dt + \eps |G'(\eps)|
    }
    = \frac12 \sum_{v \in V^\perp \minuso : g(v) \geq \eps} g(v).
  \]
  Hence
  \[
    \sum_{v \in V^\perp \minuso} |f(v)|
    \leq \br{\frac12 + \eps^{3/4}} \sum_{v \in V^\perp \minuso} g(v).
  \]

  It remains to choose $\eps$ so that $G(\eps)$ does not contain a nontrivial subgroup.
  Suppose $G(\eps)$ contains $tv$ for all $t \in \F_p$ (the prime subfield of $\F_q$), where $v \in V^\perp$, i.e.,
  \[
    \prod_{i=1}^n \nuhat_{ij}(t v_j) \geq \eps.
  \]
  By the AM--GM inequality,
  \[
    \frac1n \sum_{j=1}^n \nuhat_{ij}(tv_j)^2 \geq \eps^{2/n}.
  \]
  Let $S = \supp v$, and for each $j \in S$ let $H_j$ be the subgroup of $\F_q$ generated by $v_j$.
  Then
  \begin{align*}
    \eps^{2/n}
     & \leq \frac{|S^c|}{n} + \frac1n \sum_{j \in S} \|\nuhat|_{H_j}\|_2^2
     &                                                                        & \text{(by averaging over $t \in \F_p$)}          \\
     & = \frac{|S^c|}{n} + \frac1n \sum_{j \in S} \|\nu \bmod H_j^\perp\|_2^2
     &                                                                        & \text{(by Parseval)}                             \\
     & \leq \frac{|S^c|}{n} + \frac{|S|}{n} (1 - \alpha/8)
     &                                                                        & \text{(since $\nu_{ij}$ is $\alpha/8$-balanced)} \\
     & \leq 1-s/n + (s/n)(1-\alpha / 8)
     &                                                                        & \text{(since $|S| \geq s$)}                      \\
     & = 1 - (s/n) (\alpha /8).
  \end{align*}
  Hence $\eps \leq \exp(-\alpha s / 16)$.
  Thus we may take $\eps = \exp(-\alpha s / 17)$.
\end{proof}

\section{The local problem, part 2: correlations of eigenvalue events}
\label{sec:local-problem-part-2}

In this section we study correlations of eigenvalue events.
Assume $m > 1$ and let $\lambda_1, \dots, \lambda_m \in \F_q$.
For $M \sim \M(\bmu)$, let
\[
  R = R_{\lambda_1, \dots, \lambda_m} =
  \bigcap_{j=1}^m
  \row(M - \lambda_j)
  .
\]
From elementary linear algebra,
\begin{equation}
  \label{R-perp}
  R^\perp
  = \sum_{j=1}^m \ker(M - \lambda_j)
  = \ker (M - \lambda_1) \cdots (M - \lambda_m)
  .
\end{equation}
We will adapt the method of the previous section to prove that $R$ is suitably generic with high probability.
However, one or two aspects of the proof are now much more troublesome,
so we must make a few adjustments to the hypotheses and taxonomy:
\begin{enumerate}
  \item Assume $m$ is not too large: $m 2^m < c \alpha^2 n / \log n$.
  \item Assume $q$ is not too large: $\log q < c \alpha^2 n / (m 2^m)$.
  \item Assume $q$ is not too small: $q \geq C$ for some large constant $C$ ($C = \zeta^{-1/2}$).
  \item Call $V$ sparse only if there is some nonzero $v \perp V$ with $|\supp v| \leq C \alpha^{-1}$ for some large constant $C$.
  \item Narrow the definitions of unsaturated and semi-saturated and broaden the definition of saturated by replacing $e^{-\zeta \alpha n}$ with the larger quantity
        \[
          \exp(-\zeta (\alpha n / m)^{1/2}).
        \]
\end{enumerate}

\begin{theorem}
  \label{thm:saturation-m}
  For $M \sim \M(\bmu)$ and fixed distinct $\lambda_1, \dots, \lambda_m \in \F_q$,
  \[
    \P\br{
      R_{\lambda_1, \dots, \lambda_m}~\textup{is saturated}
    }
    = 1 - O_\alpha(\exp\br{ -c \alpha n / (m 2^m)}).
  \]
\end{theorem}

The most novel part of the argument is now the treatment of sparse vectors, which requires some reasoning about the polynomial evaluation $(M - \lambda_1) \cdots (M - \lambda_m) v$.
The unsaturated case is also more complicated and the reason we need to broaden the meaning of saturated.

\subsection{Sparse subspaces}

From \eqref{R-perp}, $v \perp R$ if and only if
\begin{equation}
  \label{M-lambda-v-condition}
  (M - \lambda_1) \cdots (M - \lambda_m) v = 0.
\end{equation}
For each fixed $v$ (sparse or not) we will show that \eqref{M-lambda-v-condition} has probability bounded by $O(\exp\br{-\alpha n / (m 2^{m+1})})$.
Since the number of $v$ of support size at most $s$ is at most $\binom{n}{s} q^s$, we may sum over all possibilities for $v$, and it follows that
\[
  \P(\exists~\text{nonzero}~ v \perp R~\text{of support}\leq s)
  \lesssim \binom{n}{s} q^s \exp(-\alpha n / (m2^{m+1})).
\]
We are assuming that $s \leq C \alpha^{-1}$ (with the value of $C$ determined by a later part of the argument), so this is exponentially negligible provided
\[
  C \alpha^{-1} (\log n + \log q) < c \alpha n / (m 2^m),
\]
as in our hypotheses.

To bound the probability of \eqref{M-lambda-v-condition},
we use a block decomposition of $M$
combined with a decoupling trick familiar from the quadratic Littlewood--Offord problem (see \cite{CTV}*{Lemma~6.3} or \cite{tao-vu}*{Section~7.6}).
To explain the trick, suppose $X$ and $Y$ are independent random variables and $f(X, Y)$ is a function taking values in a vector space, and suppose we are interested in $\P(f(X, Y) = 0)$.
Let $X'$ be a copy of $X$ independent from both $X$ and $Y$.
By Cauchy--Schwarz,
\begin{align*}
  \P(f(X, Y) = 0)
   & \leq \P(f(X, Y) = f(X', Y) = 0)^{1/2}  \\
   & \leq \P(f(X, Y) - f(X', Y) = 0)^{1/2}.
\end{align*}
In particular, applying this iteratively to a function of the form $f(X_1, \dots, X_k, Y)$,
\[
  \P(f(X_1, \dots, X_k, Y) = 0)
  \leq \P\br{
    \sum_{\omega \subset [k]}
    (-1)^{|\omega|} f(X^\omega, Y) = 0
  }^{1/2^k},
\]
where $X^\omega$ indicates $(X_1, \dots, X_k)$ but with $X_i$ replaced with $X'_i$ for $i \in \omega$.
This trick is useful for reducing a polynomial problem to a multilinear problem.

Suppose $[n]$ is partitioned into $m$ blocks of sizes $n_1, \dots, n_m \geq \floor{n/m}$.
Let the corresponding block decomposition of $M$ be
\[
  M =
  \begin{pmatrix}
    B_{11} & \cdots & B_{1m} \\
    \vdots & \ddots & \vdots \\
    B_{m1} & \cdots & B_{mm}
  \end{pmatrix}
  .
\]
Here $B_{ab}$ is an $n_a \times n_b$ matrix
with independent $\alpha$-balanced entries, and different blocks are independent.
Let $f(t) = (t - \lambda_1) \cdots (t - \lambda_m)$.
The block decomposition of $f(M)v$ is
\[
  f(M)v =
  \br{\sum_{x_1, \dots, x_m} (B_{x_0x_1} - \lambda_1 \delta_{x_0x_1}) \cdots (B_{x_{m-1} x_m} - \lambda_m \delta_{x_{m-1} x_m}) v_{x_m}
  }_{x_0=1}^m.
\]

We will treat the blocks $B_{12}, B_{23}, \dots, B_{m1}$ specially.
Let $B'_{12}, B'_{23}, \dots, B'_{m1}$ be independent copies of these blocks.
For $\omega \subset [m]$ let $M^\omega$ be equal to $M$ but with the blocks indicated by $\omega$ replaced with their primed versions, e.g.,
\[
  M^{[m]} =
  \begin{pmatrix}
    B_{11}    & B_{12}'   & B_{13}    & \cdots & B_{1m}     \\
    B_{21}    & B_{22}    & B_{23}'   & \cdots & B_{2m}     \\
              & \vdots    &           & \ddots & \vdots     \\
    B_{m-1,1} & B_{m-1,2} & B_{m-1,3} & \cdots & B_{m-1,m}' \\
    B_{m1}'   & B_{m2}    & B_{m3}    & \cdots & B_{mm}
  \end{pmatrix}.
\]
By $m$ applications of Cauchy--Schwarz,
\[
  \P(f(M)v = 0)
  \leq \P\br{\sum_{\omega\subset[m]} (-1)^{|\omega|} f(M^\omega) v = 0}^{1/2^m}.
\]
The first block of $\sum_{\omega \subset[m]} (-1)^{|\omega|} f(M^\omega) v$ is
\[
  \sum_{\omega \subset[m]} (-1)^{|\omega|} (f(M^\omega)v)_1
  = (B_{12} - B_{12}') (B_{23} - B_{23}') \cdots (B_{m1} - B_{m1}') v_1.
\]

We may assume $v_1 \neq 0$.
Each of the differences $B_{12} - B'_{12}, \dots, B_{m1} - B_{m1}'$ is again a matrix with independent $\alpha$-balanced entries.
If $(f(M)v)_1 = 0$, there is some largest index $i$ such that
\[
  (B_{i,i+1} - B'_{i,i+1}) \cdots (B_{m1} - B'_{m1}) v_1 = 0.
\]
Since different blocks are independent, it follows from $m$ applications of Odlyzko's lemma that
\[
  \P((B_{12} - B'_{12}) \cdots (B_{m1} - B'_{m1}) v_1 = 0)
  \leq \sum_{i=1}^m (1-\alpha)^{n_i}
  \leq m (1-\alpha)^{\floor{n/m}}
  .
\]
Hence
\[
  \P(f(M)v = 0)
  \leq \br{m (1-\alpha)^{\floor{n/m}}}^{1/2^m}
  \lesssim (1-\alpha)^{\floor{n/m}/2^m}.
\]
Assuming $m < n/2$, we may bound this by
\[
  \exp(-\alpha n / (m 2^{m+1})),
\]
which proves our claim.

\subsection{Semi-saturated subspaces}

The following lemma is a linear version of the Chinese remainder theorem. It will be used several times.

\begin{lemma}
  \label{CRT}
  Let $\lambda_1, \dots, \lambda_m \in \F_q$ be distinct,
  let $V_1, \dots, V_m$ be subspaces such that $V_1^\perp, \dots, V_m^\perp$ are independent,
  and let $V = V_1 \cap \cdots \cap V_m$.
  There are $x_1, \dots, x_n \in \F_q^n$ such that
  \[
    \bigcap_{j=1}^m (\lambda_j e_i + V_j)
    =
    x_i + V
    \qquad (1 \leq i \leq n).
  \]
\end{lemma}
\begin{proof}
  The linear map $\F_q^n / V \to \F_q^n / V_1 \oplus \cdots \oplus \F_q^n / V_m$ is injective by definition of $V$,
  so surjective by comparing dimensions.
\end{proof}

Suppose $V \leq \F_q^n$ is semi-saturated and $d$-codimensional.
Then
\[
  e^{-\zeta \alpha n}
  \leq \exp(-\zeta (\alpha n / m)^{1/2})
  < \rho(V)
  \leq 10 / q^d.
\]
Consider all ways of decomposing $V^\perp = V_1^\perp \oplus \cdots \oplus V_m^\perp$, and we will consider the possibility
\[
  \row(M - \lambda_j) \leq V_j  \qquad (j \in [m]).
\]
Since $V^\perp$ has dimension $d$, the number of ways of decomposing $V^\perp$ is bounded by
the number of partitions of $[d]$ into $m$ intervals times the number of ordered bases of $V^\perp$,
which is at most $2^d q^{d^2}$.
For each, we have, using Lemma~\ref{CRT},
\begin{align*}
  \P\br{X_i \in \bigcap_{j=1}^m \br{\lambda_j e_i + V_j}}
   & = \P(X_i \in x_i + V) \\
   & \leq 11 / q^d
\end{align*}
since $V$ is semi-saturated.
Hence, by independence of the rows,
\[
  \P\br{
    \bigcap_{j=1}^m \{
    \row(M - \lambda_j) \leq V_j
    \}
  } \leq (11 / q^d)^n.
\]
Hence
\[
  \P\br{
    \bigcap_{j=1}^m \row(M - \lambda_j) \leq V
  } \leq O(1)^n q^{d^2 - dn}.
\]
As in the previous section, there is some $i \in [n]$ and some $v \in V^\perp \minuso$ such that
\[
  \sum_{j=1}^n (1 - |\muhat_{ij}(v_j)|) < \zeta \alpha n,
\]
so by Markov's inequality there is a set $J \subset [n]$ of size $|J| \geq n/2$ such that
\[
  v_j \in \Spec_{1 - 2\zeta \alpha} \mu_{ij}
\]
for all $j \in J$.
By Lemma~\ref{small-spec}, assuming $q \geq \zeta^{-1/2}$ (one of our hypotheses),
\[
  |\Spec_{1-2\zeta\alpha} \mu| \leq O(\zeta^{1/2} q) + 1 = O(\zeta^{1/2} q).
\]
Hence the number of possibilities for $v$ is bounded by
\[
  n 2^n O(\zeta^{1/2} q)^{n/2} q^{n/2} = O(1)^n \zeta^{n/4} q^n,
\]
and the number of possibilities for $V$ is bounded by
\[
  O(1)^n \zeta^{n/4} q^{dn - d(d-1)}.
\]
Hence
\[
  \P\br{
    R
    ~\text{semi-saturated}
    \wedge
    \codim R = d
  }
  \leq O(1)^n \zeta^{n/4} q^d.
\]
For any semi-saturated $V$ we have $q^d = O(1)^n$,
so this is negligible provided $\zeta$ is sufficiently small.

\subsection{Unsaturated subspaces}

\def\RR{\mathcal{R}}
Finally we consider unsaturated subspaces.
Let $\UU_i$ be the set of $m$-tuples $(V_1, \dots, V_m)$ such that $V = V_1 \cap \cdots \cap V_m$ is unsaturated and $\rho(V) = \rho_i(V)$.
Let $\RR$ be the $m$-tuple $(\row(M - \lambda_j))_{j=1}^m$ whose intersection is $R$.
It suffices to show that
\begin{equation}
  \label{unsat-goal-m>1}
  \P(\RR \in \UU_i) \leq e^{-cn}
\end{equation}
for each $i$.
Here is a brief summary of the argument.
Let $\nu_{ij}$ be constructed from $\mu_{ij}$ as in Lemma~\ref{lem:nu}.
Let
\[
  \RR' = (\row(N_1), \dots, \row(N_m)),
\]
where $N_j$ is $M - \lambda_j$ but with most of the rows replaced with independent copies of $Y$, where $Y$ satisfies \eqref{Yi-law}.
Then we will show that, for any $\VV \in \UU_i$,
\[
  \P(\RR = \VV) \leq e^{-cn} \P(\RR' = \VV).
\]
We will then use the disjointness of the events $\{\RR' = \VV\}$ to infer \eqref{unsat-goal-m>1}.

Let $\VV \in \UU_i$.
Let $Y_1, \dots, Y_r$ be independent copies of $Y$.
Let $t = n - r$.
Let
\[
  B_V = \{Y_1, \dots, Y_r~\text{lin.~ind.~and in}~V\}.
\]
As before, by Odlyzko's lemma
\begin{equation}
  \label{BV-ratio-m>1}
  \P(B_V) \geq
  1 - \frac{r (1 - \alpha/8)^{t+1}}{\P(Y \in V)}.
\end{equation}
We will again ensure that this is at least $1/2$.
By independence of $X_1, \dots, X_n$ and $Y_1, \dots, Y_r$,
\begin{equation}
  \label{B-conditioning}
  \P(\RR = \VV)
  = \frac{\P(B_V \wedge \RR = \VV)}{\P(B_V)}.
\end{equation}
If $B_V$ holds and $\row(M - \lambda_j) = V_j$ for each $j$ then,
by the exchange lemma,
for each $j \in \{1, \dots, m\}$
there is a set $T_j \subset [n]$ of size $t$ such that
$\row(N_j) = V_j$, where $N_j$ is the matrix obtained from $M - \lambda_j$ by replacing the rows indexed by $T_j^c$ with $Y_1, \dots, Y_r$.
The rows indexed by $T_j^c$ must also be contained in $V_j$.
Let $T = T_1 \cup \cdots \cup T_m$.\footnote{The need to take the union here is the reason we have to broaden the definition of unsaturated,
  and ultimately the reason why $O(\exp(-cn^{1/2}))$ appears in Theorem~\ref{main} rather than $O(e^{-cn})$.}
Then
\def\SS{\Sigma}
\def\Ssubstack{{\sigma \in \SS}}
\[
  \label{B&R-bound}
  \P(B_V \wedge \RR = \VV)
  \leq \sum_\Ssubstack
  \P(\RR_\sigma = \VV)
  \prod_{i'\notin T} \P\br{\bigcap_{j=1}^m \{X_{i'} - \lambda_j e_{i'} \in V_j\}},
\]
where $\SS$ is the set of all $\sigma = (T_1, \dots, T_m)$ such that $T_j \subset[n]$ and $|T_j|=t$ for each $j$,
and
\[
  \RR_\sigma = (\row(N_1), \dots, \row(N_m)).
\]
Hence, using Lemma~\ref{CRT}, there are $x_{i'} \in \F_q^n$ for each $i' \notin T$ such that
\begin{align*}
  \P(\RR = \VV)
   & \lesssim \sum_\Ssubstack
  \P(\RR_\sigma = \VV)
  \frac{\prod_{i' \notin T} \P\br{\bigcap_{j=1}^m \{X_{i'} - \lambda_j e_{i'} \in V_j\}}}
  {\P(Y \in V)^r}
  \\
   & = \sum_\Ssubstack
  \P(\RR_\sigma = \VV)
  \frac{\prod_{i' \notin T} \P\br{X_{i'} - x_{i'} \in V}}
  {\P(Y \in V)^r}
  \\
   & \leq \sum_\Ssubstack
  \P(\RR_\sigma = \VV)
  \max_{i'} \frac{\P(X_{i'} - x_{i'} \in V)^{n - mt}}{\P(Y \in V)^r}
  .
\end{align*}
Summing over $\VV \in \UU_i$,
and assuming $mt < n/2$,
\begin{align*}
  \P(R \in \UU_i)
   & \lesssim \sum_\Ssubstack
  \P(\RR_\sigma \in \UU_i)
  \max_{\VV \in \UU_i} \theta_\VV^{n/2} / \P(Y \in V)^{(m-1)t}
  \\
   & \leq
  \binom{n}{t}^m \max_{\VV \in \UU_i} \theta_\VV^{n/2} / \P(Y \in V)^{(m-1)t}
  ,
\end{align*}
where
\[
  \theta_\VV
  = \max_{i'} \frac{\P(X_{i'} - x_{i'} \in V)}{\P(Y \in V)}.
\]
To complete the argument, we must show that, for all $\VV \in \UU_i$,
\begin{enumerate}[(a)]
  \item\label{a-m>1} $\P(Y \in V) \geq 2 r (1 - \alpha/8)^{t+1}$ (so that \eqref{BV-ratio-m>1} $\geq 1/2$);
  \item\label{a'-m>1} $\P(Y \in V)^{(m-1)t} > e^{-n/10}$;
  \item\label{b-m>1} $\theta_\VV \leq 0.6$.
\end{enumerate}

Let $\VV \in \UU_i$ and let $d = \codim V$.
By Lemma~\ref{Y-lemma},
\[
  10 / q^d
  < \rho(V) = \rho_i(V)
  \leq \br{1/2 + e^{-c\alpha s}} |\P(Y \in V) - 1/q^d|,
\]
provided that $V^\perp$ has no nonzero vector with support size less than $s$.
In particular $\P(Y \in V) > 20(1 + 2e^{-c\alpha s})^{-1} / q^d$,
and
\begin{align*}
  \P(X_{i'} - x_{i'} \in V)
   & \leq 1/q^d + \rho(V)                                                   \\
   & \leq 1/q^d + (1/2 + e^{-c\alpha s}) (\P(Y \in V) - 1/q^d))             \\
   & \leq (1/2 + e^{-c\alpha s}) (\P(Y \in V) + 1/q^d)                      \\
   & \leq (1/2 + e^{-c\alpha s})(1 + (1 + 2e^{-c\alpha s})/20) \P(Y \in V),
\end{align*}
so \ref{b-m>1} is satisfied, provided $s > C \alpha^{-1}$ for large enough $C$.
Moreover, since
\[
  \P(Y \in V) > \rho(V),
\]
\ref{a-m>1} and \ref{a'-m>1} are satisfied provided
\[
  \rho(V) \geq \max\left\{
  2 r \exp(- c \alpha t),
  \exp(- c (m-1)^{-1} t^{-1} n)
  \right\}.
\]
A good choice for $t$ is $t \asymp (\alpha^{-1} n/m)^{1/2}$ (and $r = n - t$).
Then it suffices that
\[
  \rho(V) > \exp\br{ -\zeta (\alpha n / m)^{1/2}}
\]
for sufficiently small constant $\zeta$,
as in the definition of unsaturated.

\subsection{Correlations of eigenvalue events}

In this last subsection we will use Theorem~\ref{thm:saturation-m} to prove Theorem~\ref{thm:correlation-bound}.
We assume $m < c \log n$ and
\[
  C \leq \log q
  < (\zeta / 2) (\alpha n)^{1/2} / m^{3/2}.
\]
Let $\lambda_1, \dots, \lambda_m \in \F_q$ be distinct and let $M \sim \M(\bmu)$.
Recall that $E_\lambda$ is the event that $M$ has eigenvalue $\lambda$.
We will estimate
\[
  \P(E_{\lambda_1} \cap \cdots \cap E_{\lambda_m}).
\]

Let
\[
  R = \bigcap_{j=1}^m \row(M - \lambda_j).
\]
By Theorem~\ref{thm:saturation-m},
\[
  \P(R~\text{not saturated})
  \lesssim_\alpha
  \exp\br{ -c \alpha n / (m 2^m)}.
\]
By Theorem~\ref{thm:prob-Mk-nonsingular},
\[
  \P(\codim R \geq md)
  \leq \sum_{j=1}^m \P(\codim \row(M - \lambda_j) \geq d)
  \lesssim m q^{-d}.
\]
Fix an integer $d \geq 1$ so that
\[
  \frac{d \log q}{\zeta (\alpha n)^{1/2} / m^{3/2}} \in [1/4, 1/2] .
\]
Then
\begin{equation}
  \label{R-sat-and-large}
  \P(R~\text{saturated} \wedge \codim R \leq md)
  = 1 - O_\alpha(\exp(-c \zeta (\alpha n)^{1/2} / m^{3/2})).
\end{equation}

Let $E_\lambda(v) = \{Mv = \lambda v\}$.
Since $\lambda_1, \dots, \lambda_m$ are distinct,
\[
  E_{\lambda_1} \cap \cdots \cap E_{\lambda_m}
  = \bigcup_{\substack{v_1, \dots, v_m \in \F_q^n \\
      \text{lin. ind.}}}
  E_{\lambda_1}(v_1) \cap \cdots \cap E_{\lambda_m}(v_m).
\]
Fix linearly independent $v_1, \dots, v_m \in \F_q^n$.
Let $V = v_1^\perp \cap \cdots \cap v_m^\perp$.
If the event ${E_{\lambda_1}(v_1) \cap \cdots \cap E_{\lambda_m}(v_m)}$ holds
then $R \leq V$,
so
\[
    \rho(V) \leq q^{\dim(V/R)} \rho(R).
\]
Assuming $R$ is saturated and $\codim R \leq md$, we therefore must have
\begin{equation}
  \label{eq:V-sat}
  \rho(V)
  < \exp(-(\zeta/2) (\alpha n / m)^{1/2})
\end{equation}
(a deterministic condition on $V$).
Now
\begin{align}
  E_{\lambda_1}(v_1) \cap \cdots \cap E_{\lambda_m}(v_m)
   & = \bigcap_{j=1}^m \bigcap_{i=1}^n
  \left\{ (X_i - \lambda_j e_i) \cdot v_j = 0 \right\} \\
   & = \bigcap_{i=1}^n
  \left\{X_i \in \bigcap_{j=1}^m \br{\lambda_j e_i + v_j^\perp} \right\}.
  \label{Evi-rows}
\end{align}
By Lemma~\ref{CRT} there are $x_1, \dots, x_n \in \F_q^n$ (depending on $\lambda_1, \dots, \lambda_m, v_1, \dots, v_m$) such that
\[
  \bigcap_{j=1}^m \br{\lambda_j e_i + v_j^\perp}
  = x_i + V.
\]
Hence
\begin{align*}
  \P\br{X_i \in \bigcap_{j=1}^m \br{\lambda_j e_i + v_j^\perp}}
   & = \P(X_i \in x_i + V) \\
   & = 1/q^m + O(\rho(V)).
\end{align*}
Thus by \eqref{Evi-rows} and independence of $X_1, \dots, X_n$,
\[
  \P\br{E_{\lambda_1}(v_1) \cap \cdots \cap E_{\lambda_m}(v_m)}
  = (1/q^m + O(\rho(V)))^n
  = q^{-mn} \exp O(q^m \rho(V) n).
\]
Summing over all linearly independent $v_1, \dots, v_m$ up to scale,
ignoring those not satisfying \eqref{eq:V-sat},
and accounting for the possibility that $R$ is either not saturated
or of codimension at least $m d$ using \eqref{R-sat-and-large},
it follows that
\begin{align*}
  \P(E_{\lambda_1} \cap \cdots \cap E_{\lambda_m})
   & \leq \frac{(q^n - 1)^m}{(q-1)^m} q^{-mn} \exp O(\exp(-c \zeta (\alpha n / m)^{1/2}) n) \\
   & \qquad + O_\alpha(\exp(-c\zeta (\alpha n)^{1/2} / m^{3/2})                             \\
   & = (q-1)^{-m}
  + O_\alpha(\exp(-c \zeta (\alpha n)^{1/2} / m^{3/2}))
  .
\end{align*}
This finishes the proof of Theorem~\ref{thm:correlation-bound}.

\bibliography{refs}

\begin{bibdiv}
\begin{biblist}

\bib{ABT}{article}{
      author={Arratia, Richard},
      author={Barbour, A.~D.},
      author={Tavar\'{e}, Simon},
       title={Limits of logarithmic combinatorial structures},
        date={2000},
        ISSN={0091-1798},
     journal={Ann. Probab.},
      volume={28},
      number={4},
       pages={1620\ndash 1644},
         url={https://doi.org/10.1214/aop/1019160500},
      review={\MR{1813836}},
}

\bib{ABT-book}{book}{
      author={Arratia, Richard},
      author={Barbour, A.~D.},
      author={Tavar\'{e}, Simon},
       title={Logarithmic combinatorial structures: a probabilistic approach},
      series={EMS Monographs in Mathematics},
   publisher={European Mathematical Society (EMS), Z\"{u}rich},
        date={2003},
        ISBN={3-03719-000-0},
         url={https://doi.org/10.4171/000},
      review={\MR{2032426}},
}

\bib{BSK}{article}{
      author={{Bary-Soroker}, Lior},
      author={Kozma, Gady},
       title={Irreducible polynomials of bounded height},
        date={2020},
        ISSN={0012-7094},
     journal={Duke Math. J.},
      volume={169},
      number={4},
       pages={579\ndash 598},
         url={https://doi.org/10.1215/00127094-2019-0047},
      review={\MR{4072635}},
}

\bib{BSKK}{article}{
      author={{Bary-Soroker}, Lior},
      author={{Koukoulopoulos}, Dimitris},
      author={{Kozma}, Gady},
       title={{Irreducibility of random polynomials: general measures}},
        date={2020-07},
     journal={arXiv e-prints},
       pages={arXiv:2007.14567},
      eprint={2007.14567},
}

\bib{BV}{article}{
      author={Breuillard, Emmanuel},
      author={Varj\'{u}, P\'{e}ter~P.},
       title={Irreducibility of random polynomials of large degree},
        date={2019},
        ISSN={0001-5962},
     journal={Acta Math.},
      volume={223},
      number={2},
       pages={195\ndash 249},
         url={https://doi.org/10.4310/acta.2019.v223.n2.a1},
      review={\MR{4047924}},
}

\bib{bourgain-vu-wood}{article}{
      author={Bourgain, Jean},
      author={Vu, Van~H.},
      author={Wood, Philip~Matchett},
       title={On the singularity probability of discrete random matrices},
        date={2010},
        ISSN={0022-1236},
     journal={J. Funct. Anal.},
      volume={258},
      number={2},
       pages={559\ndash 603},
         url={https://doi.org/10.1016/j.jfa.2009.04.016},
      review={\MR{2557947}},
}

\bib{CJMS}{article}{
      author={{Campos}, Marcelo},
      author={{Jenssen}, Matthew},
      author={{Michelen}, Marcus},
      author={{Sahasrabudhe}, Julian},
       title={{The singularity probability of a random symmetric matrix is
  exponentially small}},
        date={2021-05},
     journal={arXiv e-prints},
       pages={arXiv:2105.11384},
      eprint={2105.11384},
}

\bib{CM3}{misc}{
      author={Campos, Marcelo},
      author={Mattos, Letícia},
      author={Morris, Robert},
      author={Morrison, Natasha},
       title={On the singularity of random symmetric matrices},
        date={2019},
}

\bib{CTV}{article}{
      author={Costello, Kevin~P.},
      author={Tao, Terence},
      author={Vu, Van},
       title={Random symmetric matrices are almost surely nonsingular},
        date={2006},
        ISSN={0012-7094},
     journal={Duke Math. J.},
      volume={135},
      number={2},
       pages={395\ndash 413},
         url={https://doi.org/10.1215/S0012-7094-06-13527-5},
      review={\MR{2267289}},
}

\bib{dixon--mortimer}{book}{
      author={Dixon, John~D.},
      author={Mortimer, Brian},
       title={Permutation groups},
      series={Graduate Texts in Mathematics},
   publisher={Springer-Verlag, New York},
        date={1996},
      volume={163},
        ISBN={0-387-94599-7},
         url={https://doi.org/10.1007/978-1-4612-0731-3},
      review={\MR{1409812}},
}

\bib{EFG-invariable}{article}{
      author={Eberhard, Sean},
      author={Ford, Kevin},
      author={Green, Ben},
       title={Invariable generation of the symmetric group},
        date={2017},
        ISSN={0012-7094},
     journal={Duke Math. J.},
      volume={166},
      number={8},
       pages={1573\ndash 1590},
         url={https://doi.org/10.1215/00127094-0000007X},
      review={\MR{3659942}},
}

\bib{ferber--jain}{article}{
      author={Ferber, Asaf},
      author={Jain, Vishesh},
       title={Singularity of random symmetric matrices---a combinatorial
  approach to improved bounds},
        date={2019},
     journal={Forum Math. Sigma},
      volume={7},
       pages={e22, 29},
         url={https://doi.org/10.1017/fms.2019.21},
      review={\MR{3993806}},
}

\bib{FJSS}{article}{
      author={{Ferber}, Asaf},
      author={{Jain}, Vishesh},
      author={{Sah}, Ashwin},
      author={{Sawhney}, Mehtaab},
       title={{Random symmetric matrices: rank distribution and irreducibility
  of the characteristic polynomial}},
        date={2021-06},
     journal={arXiv e-prints},
       pages={arXiv:2106.04049},
      eprint={2106.04049},
}

\bib{horn-johnson}{book}{
      author={Horn, Roger~A.},
      author={Johnson, Charles~R.},
       title={Matrix analysis},
     edition={Second},
   publisher={Cambridge University Press, Cambridge},
        date={2013},
        ISBN={978-0-521-54823-6},
      review={\MR{2978290}},
}

\bib{hansen-schmutz}{article}{
      author={Hansen, Jennie~C.},
      author={Schmutz, Eric},
       title={How random is the characteristic polynomial of a random matrix?},
        date={1993},
        ISSN={0305-0041},
     journal={Math. Proc. Cambridge Philos. Soc.},
      volume={114},
      number={3},
       pages={507\ndash 515},
         url={https://doi.org/10.1017/S0305004100071796},
      review={\MR{1235998}},
}

\bib{komlos}{article}{
      author={Koml\'{o}s, J.},
       title={On the determinant of {$(0,\,1)$} matrices},
        date={1967},
        ISSN={0081-6906},
     journal={Studia Sci. Math. Hungar.},
      volume={2},
       pages={7\ndash 21},
      review={\MR{221962}},
}

\bib{LMN}{article}{
      author={{Luh}, Kyle},
      author={{Meehan}, Sean},
      author={{Nguyen}, Hoi~H.},
       title={{Some new results in random matrices over finite fields}},
        date={2019-07},
     journal={arXiv e-prints},
       pages={arXiv:1907.02575},
}

\bib{l-o'r-simple-spec}{article}{
      author={{Luh}, Kyle},
      author={{O'Rourke}, Sean},
       title={{Eigenvectors and controllability of non-Hermitian random
  matrices and directed graphs}},
        date={2020-04},
     journal={arXiv e-prints},
       pages={arXiv:2004.10543},
}

\bib{maples}{misc}{
      author={Maples, Kenneth},
       title={Singularity of random matrices over finite fields},
        date={2010},
}

\bib{neumann}{incollection}{
      author={Neumann, Peter~M.},
       title={Helmut {W}ielandt on permutation groups},
   booktitle={Helmut {W}ielandt mathematische werke / mathematical works,
  volume 1, group theory},
   publisher={{DE} {GRUYTER}},
       pages={3\ndash 20},
         url={https://doi.org/10.1515/9783110863383.3},
}

\bib{nguyen-paquette}{article}{
      author={Nguyen, Hoi~H.},
      author={Paquette, Elliot},
       title={Surjectivity of near-square random matrices},
        date={2020},
        ISSN={0963-5483},
     journal={Combin. Probab. Comput.},
      volume={29},
      number={2},
       pages={267\ndash 292},
         url={https://doi.org/10.1017/s0963548319000348},
      review={\MR{4079637}},
}

\bib{OP}{article}{
      author={Odlyzko, A.~M.},
      author={Poonen, B.},
       title={Zeros of polynomials with {$0,1$} coefficients},
        date={1993},
        ISSN={0013-8584},
     journal={Enseign. Math. (2)},
      volume={39},
      number={3-4},
       pages={317\ndash 348},
      review={\MR{1252071}},
}

\bib{O'RW}{article}{
      author={O'Rourke, Sean},
      author={Wood, Philip~Matchett},
       title={Low-degree factors of random polynomials},
        date={2019},
        ISSN={0894-9840},
     journal={J. Theoret. Probab.},
      volume={32},
      number={2},
       pages={1076\ndash 1104},
         url={https://doi.org/10.1007/s10959-018-0839-8},
      review={\MR{3959638}},
}

\bib{pemantle-peres-rivin}{article}{
      author={Pemantle, Robin},
      author={Peres, Yuval},
      author={Rivin, Igor},
       title={Four random permutations conjugated by an adversary generate
  {$\mathcal{S}_n$} with high probability},
        date={2016},
        ISSN={1042-9832},
     journal={Random Structures Algorithms},
      volume={49},
      number={3},
       pages={409\ndash 428},
         url={https://doi.org/10.1002/rsa.20632},
      review={\MR{3545822}},
}

\bib{PPR-invariable}{article}{
      author={Pemantle, Robin},
      author={Peres, Yuval},
      author={Rivin, Igor},
       title={Four random permutations conjugated by an adversary generate
  {$S_n$} with high probability},
        date={2016},
        ISSN={1042-9832},
     journal={Random Structures Algorithms},
      volume={49},
      number={3},
       pages={409\ndash 428},
         url={https://doi.org/10.1002/rsa.20632},
      review={\MR{3545822}},
}

\bib{reiner}{article}{
      author={Reiner, Irving},
       title={On the number of matrices with given characteristic polynomial},
        date={1961},
        ISSN={0019-2082},
     journal={Illinois J. Math.},
      volume={5},
       pages={324\ndash 329},
         url={http://projecteuclid.org/euclid.ijm/1255629830},
      review={\MR{139621}},
}

\bib{silverstein}{article}{
      author={Silverstein, Jack~W.},
       title={The spectral radii and norms of large-dimensional non-central
  random matrices},
        date={1994},
        ISSN={0882-0287},
     journal={Comm. Statist. Stochastic Models},
      volume={10},
      number={3},
       pages={525\ndash 532},
         url={https://doi.org/10.1080/15326349408807308},
      review={\MR{1284550}},
}

\bib{tikhomirov}{article}{
      author={Tikhomirov, Konstantin},
       title={Singularity of random {B}ernoulli matrices},
        date={2020},
        ISSN={0003-486X},
     journal={Ann. of Math. (2)},
      volume={191},
      number={2},
       pages={593\ndash 634},
         url={https://doi.org/10.4007/annals.2020.191.2.6},
      review={\MR{4076632}},
}

\bib{tao-vu-simple-spec}{article}{
      author={Tao, Terence},
      author={Vu, Van},
       title={Random matrices have simple spectrum},
        date={2017},
        ISSN={0209-9683},
     journal={Combinatorica},
      volume={37},
      number={3},
       pages={539\ndash 553},
         url={https://doi.org/10.1007/s00493-016-3363-4},
      review={\MR{3666791}},
}

\bib{tao-vu-singularity}{article}{
      author={Tao, Terence},
      author={Vu, Van~H.},
       title={On random $\pm$1 matrices: Singularity and determinant},
        date={2005},
     journal={Random Structures and Algorithms},
      volume={28},
      number={1},
       pages={1\ndash 23},
         url={https://doi.org/10.1002/rsa.20109},
}

\bib{tao-vu}{book}{
      author={Tao, Terence},
      author={Vu, Van~H.},
       title={Additive combinatorics},
      series={Cambridge Studies in Advanced Mathematics},
   publisher={Cambridge University Press, Cambridge},
        date={2010},
      volume={105},
        ISBN={978-0-521-13656-3},
        note={Paperback edition [of MR2289012]},
      review={\MR{2573797}},
}

\bib{vershynin-book}{book}{
      author={Vershynin, Roman},
       title={High-dimensional probability},
      series={Cambridge Series in Statistical and Probabilistic Mathematics},
   publisher={Cambridge University Press, Cambridge},
        date={2018},
      volume={47},
        ISBN={978-1-108-41519-4},
         url={https://doi.org/10.1017/9781108231596},
        note={An introduction with applications in data science, With a
  foreword by Sara van de Geer},
      review={\MR{3837109}},
}

\bib{vu-survey}{article}{
      author={{Vu}, Van},
       title={{Recent progress in combinatorial random matrix theory}},
        date={2020-05},
     journal={arXiv e-prints},
       pages={arXiv:2005.02797},
}

\end{biblist}
\end{bibdiv}
\end{document}